\documentclass{article}
\usepackage{fullpage}
\usepackage{subcaption}
\usepackage{amsthm,amsmath,bbm}
\theoremstyle{plain}
\newtheorem{theorem}{Theorem}[section]

\newtheorem{lemma}[theorem]{Lemma}

\theoremstyle{definition}
\newtheorem{definition}[theorem]{Definition}
\theoremstyle{remark}

\newtheorem{example}[theorem]{Example}

\newcommand\ifwithlilypond[2]{#2}
\newcommand\ifwithps[2]{#2}

\usepackage[utf8]{inputenc}
\usepackage{palatino}
\usepackage{amssymb}

\usepackage[table]{xcolor}
\usepackage{graphics}
\usepackage{pst-plot}
\ifwithps{\usepackage{tikz}\usetikzlibrary{turtle}}{}
\usepackage{multicol}
\usepackage{apptools,chngcntr} %per numerar lemmes appendix

\renewcommand\leq\leqslant
\renewcommand\geq\geqslant
\newcommand\mut[1]{\ignorespaces}

\usepackage{alltt}
\newcommand\citet[1]{\cite{#1}}

\begin{document}

\title{Tempered Monoids of Real Numbers, the Golden Fractal Monoid, and the Well-Tempered Harmonic Semigroup \\ {\small Published in Semigroup Forum, Springer, vol. 99, n. 2, pp. 496-516, November 2019. ISSN: 0037-1912.} }

\author{Maria Bras-Amorós}
%\institute{Universitat Rovira i Virgili, Tarragona, Catalonia, \email{maria.bras@urv.cat}}

\maketitle

\begin{abstract}
This paper deals with the algebraic structure of the sequence of harmonics when combined with equal temperaments. Fractals and the golden ratio appear surprisingly on the way.

The sequence of physical harmonics is an increasingly enumerable submonoid of $({\mathbb R}^+,+)$ whose pairs of consecutive terms get arbitrarily close as they grow. These properties suggest the definition of a new mathematical object which we denote a {\em tempered monoid}. Mapping the elements of the tempered monoid of physical harmonics from ${\mathbb R}$ to ${\mathbb N}$ may be considered tantamount to defining equal temperaments. The number of equal parts of the octave in an equal temperament corresponds to the multiplicity of the related numerical semigroup.
  
Analyzing the sequence of musical harmonics we derive two important properties that tempered monoids may have: that of being product-compatible and that of being fractal. We demonstrate that, up to normalization, there is only one product-compatible tempered monoid, which is the logarithmic monoid, and there is only one nonbisectional fractal monoid which is generated by the golden ratio. 

The example of half-closed cylindrical pipes imposes a third property to the sequence of musical harmonics, the so-called odd-filterability property.

We prove that the maximum number of equal divisions of the octave such that the discretizations of the golden fractal monoid and the logarithmic monoid coincide, and such that the discretization is odd-filterable is $12$. This is nothing else but the number of equal divisions of the octave in classical Western music.
\end{abstract}


{\bf Classification:} \textit{2010 Mathematics Subject Classification}: 00A65; 20M14

\section{Introduction}
%When a fundamental tone is played, its harmonics are also heard, as has been studied by the science of acoustics.
The science of acoustics describes how harmonics of a given fundamental tone appear together with the fundamental tone when this one is played by a mechanical musical instrument. The way harmonics arise in each instrument describes its timbre. This paper deals with the algebraic structure of the sequence of harmonics when combined with equal temperaments. Fractals and the golden ratio appear surprisingly on the way.

\subsection*{Hidden fractal patterns in music} Fractal geometry, coined by Mandelbrot \cite{Mandelbrot75,Mandelbrot82},
studies self-similarity, appearing when each small piece of a shape contains a scaled copy of the whole shape, and, in turn, each small piece of this scaled copy contains an even smaller scaled copy of the whole shape, and so on. This self-similarity pattern has then been observed in many other fields apart from geometry. Music has not been an exception \cite{VossClarke,HsuHsu,Gardner,Madden,Beran}. Brothers, among his vast literature (see \cite{Brothers} and all the references in the chapter) identifies several ways of scaling in music: duration, pitch, melodic intervals, melodic moments, harmonic intervals, structure, and melodic or rhythmic motivic scaling. This way, fractal patterns are identified in scores by Bach, Mozart, Ravel, Chopin, Beethoven, Strauss, Debussy, and even The Beatles. 
%We particularly like the example we found in Figaro's melody in the aria {\em Non più andrai} of {\em Le Nozze di Figaro} (see Figure~\ref{mozart}).
%\begin{figure}
%  \begin{center}
%    \ifwithlilypond{\lilypondfile[noindent,staffsize=17,line-width=140,system-system-spacing = 24]{fractalpitchmozart.ly}}{\resizebox{.8\textwidth}{!}{\includegraphics{fractalpitchmozart.jpg}}}
%\end{center}
%\caption{Fractals in Mozart's aria {\em Non più andrai} of {\em Le Nozze di Figaro}.
%In the first three notes we have two short notes at the same pitch and a longer one slightly below. Each of the three groups of notes in the beginning have the same structure. If we now look each of these groups of three notes as a block each, then the blocks turn on the same idea but on a larger scale. Next, the same drawing is repeated one note below. After that, the first three notes make an arpeggio or jump up. The same goes for the three notes that continue and the other three. Also, if we group these three-note groups into one block each, the three blocks make the same drawing on a larger scale. 
%In the second part of the second line, when we still have the ascending line in our memory, the melody is descending, at a large scale. Thus, on a large scale we have a ``bid-and-low'' drawing, which is the same as every three-note group in this descending part. The described patterns are not indefinitely repeated at smaller and smaller scales as should be in a perfect fractal, just as our ear is not indefinitely sensitive.}
%\label{mozart}
%\end{figure}
Wuorinen
\cite{Wuorinen}
relaxes the notion of self-similarity to self-affinity and
specifies three levels at which the fractal characteristics of music seem to manifest themselves: the acoustic signal, pitch and rhythm, and structure. The three levels form a progression from the concrete-physical to the metaphorical-structural, and thus, increasing in artistic significance while diminishing in specificity. On his side, Madden \cite{Madden} points at fractal behaviour of the harmonic sequence but reducing it to its logarithmic behaviour. Our discussion starts there. Wuorinen \cite{Wuorinen} postulates that although the musical sound has been traditionally divided into pitch, rhythm, timbre and loudness, only pitch and rhythm can be organized in a fractal way. He asserts that ``{\em timbre and loudness seem not to have fractal characteristics as they figure in music}''.
Frame and Urry
\cite{FrameUrry}, state that ``{\em fractal aspects of music are patterns hidden by our sequential perception of music. What other fractal patterns will music reveal?}" We can say that in this paper we reveal one of such fractal patterns mentioned by Frame and Urry \cite{FrameUrry}. Our pattern is related to the algebraic structure of the sequence of harmonics when combined with equal temperaments, and so, indirectly related to timbre as addressed by Wuorinen.

\subsection*{The golden ratio} The golden ratio is, if not the most, one of the most popular irrational numbers. It appears in nature and in art and it is strongly related to the equally popular Fibonacci numbers. The literature about the golden ratio, also called the divine proportion, is enormous and it is not our aim to be exhaustive. We just mention the book \cite{Livio} and refer the reader to all the references therein. Music analysts speculate about the presence of the divine proportion in many compositions of several authors such as Bach and Mozart, or more insistently in Béla Bartók's scores. Certainly, many 20th and 21st century composers have used it consciously in their compositions. Our fractal patterns will be indeed generated by the golden ratio.

\subsection*{The harmonic sequence}
Let us return to the phenomenon of harmonics. Tones are essentially wave frequencies and the harmonics of a fundamental tone correspond to integer multiples of that frequency. 
\begin{figure}
  \begin{center}
    \ifwithps{\resizebox{\textwidth}{!}{\begin{tabular}{lll}
string & cylindrical open pipe & cylindrical closed pipe
\\
\begin{pspicture}(0,-1)(5,.5)
  \psaxes[yAxis=false,labels=none,ticksize=0,linewidth=.05](0,0)(4,0)
  \psaxes[yAxis=false,labels=none,ticksize=0,linewidth=.01]{<->}(0,-.5)(4,-.5)
\rput(2,-.65){\scalebox{.5}{$L$}}
\end{pspicture}
&
\begin{pspicture}(0,-1)(5,.5)
\psaxes[yAxis=false,labels=none,ticksize=0,linewidth=.05](0,-.25)(4,-.25)
\psaxes[yAxis=false,labels=none,ticksize=0,linewidth=.05](0,.25)(4,.25)
\psaxes[yAxis=false,labels=none,ticksize=0,linewidth=.01]{<->}(0,-.5)(4,-.5)
\rput(2,-.65){\scalebox{.5}{$L$}}
\end{pspicture}
&
\begin{pspicture}(0,-1)(5,.5)
\psaxes[yAxis=false,labels=none,ticksize=0,linewidth=.05](0,-.25)(4,-.25)
\psaxes[yAxis=false,labels=none,ticksize=0,linewidth=.05](0,.25)(4,.25)
\psaxes[xAxis=false,labels=none,ticksize=0,linewidth=.05](0,-.25)(0,.25)
\psaxes[yAxis=false,labels=none,ticksize=0,linewidth=.01]{<->}(0,-.5)(4,-.5)
\rput(2,-.65){\scalebox{.5}{$L$}}
\end{pspicture}
\\  
%amplitude of &amplitude of &amplitude of\\
string vibration &
motion of the air &
motion of the air
\\
\begin{pspicture}(0,-.5)(5,.5)
\rput[l](4.05,.15){\scalebox{.75}{$\nu=\nu_0$}}
\rput[l](4.05,-.15){\scalebox{.75}{$\lambda_0=2L$}}
\psaxes[yAxis=false,labels=none,Dx=4,ticksize=-.5 .5,linecolor=gray,linewidth=.02,linestyle=dashed](0,0)(4,0)
\psplot[plotstyle=curve,plotpoints=360]{0}{4}{45 x mul sin 2 div}
\psplot[plotstyle=curve,plotpoints=360]{0}{4}{0 45 x mul sin sub 2 div}
\end{pspicture}
&
\begin{pspicture}(0,-.5)(5,.5)
  \rput[l](4.05,.15){\scalebox{.75}{$\nu=\nu_0$}}
  \rput[l](4.05,-.15){\scalebox{.75}{$\lambda_0=2L$}}
\psaxes[yAxis=false,labels=none,Dx=4,ticksize=-.5 .5,linecolor=gray,linewidth=.02,linestyle=dashed](0,0)(4,0)
\psplot[plotstyle=curve,plotpoints=360]{0}{4}{45 x mul cos 2 div}
\psplot[plotstyle=curve,plotpoints=360]{0}{4}{0 45 x mul cos sub 2 div}
\end{pspicture}
&
\begin{pspicture}(0,-.5)(5,.5)
  \rput[l](4.05,.15){\scalebox{.75}{$\nu=\nu_0$}}
  \rput[l](4.05,-.15){\scalebox{.75}{$\lambda_0=4L$}}
\psaxes[yAxis=false,labels=none,Dx=4,ticksize=-.5 .5,linecolor=gray,linewidth=.02,linestyle=dashed](0,0)(4,0)
\psplot[plotstyle=curve,plotpoints=360]{0}{4}{22.5 x mul sin 2 div}
\psplot[plotstyle=curve,plotpoints=360]{0}{4}{0 22.5 x mul sin sub 2 div}
\end{pspicture}
\\
\begin{pspicture}(0,-.5)(5,.5)
  \rput[l](4.05,.15){\scalebox{.75}{$\nu_1=2\nu_0$}}
    \rput[l](4.05,-.15){\scalebox{.75}{$\lambda_1=L$}}
\psaxes[yAxis=false,labels=none,Dx=4,ticksize=-.5 .5,linecolor=gray,linewidth=.02,linestyle=dashed](0,0)(4,0)
\psplot[plotstyle=curve,plotpoints=360]{0}{4}{90 x mul sin 2 div}
\psplot[plotstyle=curve,plotpoints=360]{0}{4}{0 90 x mul sin sub 2 div}
\end{pspicture}
&
\begin{pspicture}(0,-.5)(5,.5)
  \rput[l](4.05,.15){\scalebox{.75}{$\nu_1=2\nu_0$}}
      \rput[l](4.05,-.15){\scalebox{.75}{$\lambda_1=L$}}
\psaxes[yAxis=false,labels=none,Dx=4,ticksize=-.5 .5,linecolor=gray,linewidth=.02,linestyle=dashed](0,0)(4,0)
\psplot[plotstyle=curve,plotpoints=360]{0}{4}{90 x mul cos 2 div}
\psplot[plotstyle=curve,plotpoints=360]{0}{4}{0 90 x mul cos sub 2 div}
\end{pspicture}
&
\begin{pspicture}(0,-.5)(5,.5)
  \rput[l](4.05,.15){\scalebox{.75}{$\nu_1=3\nu_0$}}
      \rput[l](4.05,-.15){\scalebox{.75}{$\lambda_1=4L/3$}}
\psaxes[yAxis=false,labels=none,Dx=4,ticksize=-.5 .5,linecolor=gray,linewidth=.02,linestyle=dashed](0,0)(4,0)
\psplot[plotstyle=curve,plotpoints=360]{0}{4}{67.5 x mul sin 2 div}
\psplot[plotstyle=curve,plotpoints=360]{0}{4}{0 67.5 x mul sin sub 2 div}
\end{pspicture}
\\
\begin{pspicture}(0,-.5)(5,.5)
\rput[l](4.05,.15){\scalebox{.75}{$\nu_2=3\nu_0$}}
\rput[l](4.05,-.15){\scalebox{.75}{$\lambda_2=2L/3$}}
\psaxes[yAxis=false,labels=none,Dx=4,ticksize=-.5 .5,linecolor=gray,linewidth=.02,linestyle=dashed](0,0)(4,0)
\psplot[plotstyle=curve,plotpoints=360]{0}{4}{135 x mul sin 2 div}
\psplot[plotstyle=curve,plotpoints=360]{0}{4}{0 135 x mul sin sub 2 div}
\end{pspicture}
&
\begin{pspicture}(0,-.5)(5,.5)
\rput[l](4.05,.15){\scalebox{.75}{$\nu_2=3\nu_0$}}
\rput[l](4.05,-.15){\scalebox{.75}{$\lambda_2=2L/3$}}
\psaxes[yAxis=false,labels=none,Dx=4,ticksize=-.5 .5,linecolor=gray,linewidth=.02,linestyle=dashed](0,0)(4,0)
\psplot[plotstyle=curve,plotpoints=360]{0}{4}{135 x mul cos 2 div}
\psplot[plotstyle=curve,plotpoints=360]{0}{4}{0 135 x mul cos sub 2 div}
\end{pspicture}
&
\begin{pspicture}(0,-.5)(5,.5)
\rput[l](4.05,.15){\scalebox{.75}{$\nu_2=5\nu_0$}}
\rput[l](4.05,-.15){\scalebox{.75}{$\lambda_2=4L/5$}}
\psaxes[yAxis=false,labels=none,Dx=4,ticksize=-.5 .5,linecolor=gray,linewidth=.02,linestyle=dashed](0,0)(4,0)
\psplot[plotstyle=curve,plotpoints=360]{0}{4}{112.5 x mul sin 2 div}
\psplot[plotstyle=curve,plotpoints=360]{0}{4}{0 112.5 x mul sin sub 2 div}
\end{pspicture}
\\
\begin{pspicture}(0,-.5)(5,.5)
\rput[l](4.05,.15){\scalebox{.75}{$\nu_3=4\nu_0$}}
\rput[l](4.05,-.15){\scalebox{.75}{$\lambda_3=L/2$}}
\psaxes[yAxis=false,labels=none,Dx=4,ticksize=-.5 .5,linecolor=gray,linewidth=.02,linestyle=dashed](0,0)(4,0)
\psplot[plotstyle=curve,plotpoints=360]{0}{4}{180 x mul sin 2 div}
\psplot[plotstyle=curve,plotpoints=360]{0}{4}{0 180 x mul sin sub 2 div}
\end{pspicture}
&
\begin{pspicture}(0,-.5)(5,.5)
\rput[l](4.05,.15){\scalebox{.75}{$\nu_3=4\nu_0$}}
\rput[l](4.05,-.15){\scalebox{.75}{$\lambda_3=L/2$}}
\psaxes[yAxis=false,labels=none,Dx=4,ticksize=-.5 .5,linecolor=gray,linewidth=.02,linestyle=dashed](0,0)(4,0)
\psplot[plotstyle=curve,plotpoints=360]{0}{4}{180 x mul cos 2 div}
\psplot[plotstyle=curve,plotpoints=360]{0}{4}{0 180 x mul cos sub 2 div}
\end{pspicture}
&
\begin{pspicture}(0,-.5)(5,.5)
\rput[l](4.05,.15){\scalebox{.75}{$\nu_3=7\nu_0$}}
\rput[l](4.05,-.15){\scalebox{.75}{$\lambda_3=4L/7$}}
\psaxes[yAxis=false,labels=none,Dx=4,ticksize=-.5 .5,linecolor=gray,linewidth=.02,linestyle=dashed](0,0)(4,0)
\psplot[plotstyle=curve,plotpoints=360]{0}{4}{157.5 x mul sin 2 div}
\psplot[plotstyle=curve,plotpoints=360]{0}{4}{0 157.5 x mul sin sub 2 div}
\end{pspicture}
\\
\mut{
\begin{pspicture}(0,-.5)(5,.5)
\rput[l](4.05,.15){\scalebox{.75}{$\nu_4=5\nu_0$}}
\rput[l](4.05,-.15){\scalebox{.75}{$\lambda_4=2L/5$}}
\psaxes[yAxis=false,labels=none,Dx=4,ticksize=-.5 .5,linecolor=gray,linewidth=.02,linestyle=dashed](0,0)(4,0)
\psplot[plotstyle=curve,plotpoints=360]{0}{4}{225 x mul sin 2 div}
\psplot[plotstyle=curve,plotpoints=360]{0}{4}{0 225 x mul sin sub 2 div}
\end{pspicture}
&
\begin{pspicture}(0,-.5)(5,.5)
\rput[l](4.05,.15){\scalebox{.75}{$\nu_4=5\nu_0$}}
\rput[l](4.05,-.15){\scalebox{.75}{$\lambda_4=2L/5$}}
\psaxes[yAxis=false,labels=none,Dx=4,ticksize=-.5 .5,linecolor=gray,linewidth=.02,linestyle=dashed](0,0)(4,0)
\psplot[plotstyle=curve,plotpoints=360]{0}{4}{225 x mul cos 2 div}
\psplot[plotstyle=curve,plotpoints=360]{0}{4}{0 225 x mul cos sub 2 div}
\end{pspicture}
&
\begin{pspicture}(0,-.5)(5,.5)
\rput[l](4.05,.15){\scalebox{.75}{$\nu_4=9\nu_0$}}
\rput[l](4.05,-.15){\scalebox{.75}{$\lambda_4=4L/9$}}
\psaxes[yAxis=false,labels=none,Dx=4,ticksize=-.5 .5,linecolor=gray,linewidth=.02,linestyle=dashed](0,0)(4,0)
\psplot[plotstyle=curve,plotpoints=360]{0}{4}{202.5 x mul sin 2 div}
\psplot[plotstyle=curve,plotpoints=360]{0}{4}{0 202.5 x mul sin sub 2 div}
\end{pspicture}
\\
\begin{pspicture}(0,-.5)(5,.5)
\rput[l](4.05,.15){\scalebox{.75}{$\nu_5=6\nu_0$}}
\rput[l](4.05,-.15){\scalebox{.75}{$\lambda_5=L/3$}}
\psaxes[yAxis=false,labels=none,Dx=4,ticksize=-.5 .5,linecolor=gray,linewidth=.02,linestyle=dashed](0,0)(4,0)
\psplot[plotstyle=curve,plotpoints=360]{0}{4}{270 x mul sin 2 div}
\psplot[plotstyle=curve,plotpoints=360]{0}{4}{0 270 x mul sin sub 2 div}
\end{pspicture}
&
\begin{pspicture}(0,-.5)(5,.5)
\rput[l](4.05,.15){\scalebox{.75}{$\nu_5=6\nu_0$}}
\rput[l](4.05,-.15){\scalebox{.75}{$\lambda_5=L/3$}}
\psaxes[yAxis=false,labels=none,Dx=4,ticksize=-.5 .5,linecolor=gray,linewidth=.02,linestyle=dashed](0,0)(4,0)
\psplot[plotstyle=curve,plotpoints=360]{0}{4}{270 x mul cos 2 div}
\psplot[plotstyle=curve,plotpoints=360]{0}{4}{0 270 x mul cos sub 2 div}
\end{pspicture}
&
\begin{pspicture}(0,-.5)(5,.5)
\rput[l](4.05,.15){\scalebox{.75}{$\nu_5=11\nu_0$}}
\rput[l](4.05,-.15){\scalebox{.75}{$\lambda_5=4L/11$}}
\psaxes[yAxis=false,labels=none,Dx=4,ticksize=-.5 .5,linecolor=gray,linewidth=.02,linestyle=dashed](0,0)(4,0)
\psplot[plotstyle=curve,plotpoints=360]{0}{4}{247.5 x mul sin 2 div}
\psplot[plotstyle=curve,plotpoints=360]{0}{4}{0 247.5 x mul sin sub 2 div}
\end{pspicture}
\\
\begin{pspicture}(0,-.5)(5,.5)
\rput[l](4.05,.15){\scalebox{.75}{$\nu_6=7\nu_0$}}
\rput[l](4.05,-.15){\scalebox{.75}{$\lambda_6=2L/7$}}
\psaxes[yAxis=false,labels=none,Dx=4,ticksize=-.5 .5,linecolor=gray,linewidth=.02,linestyle=dashed](0,0)(4,0)
\psplot[plotstyle=curve,plotpoints=360]{0}{4}{315 x mul sin 2 div}
\psplot[plotstyle=curve,plotpoints=360]{0}{4}{0 315 x mul sin sub 2 div}
\end{pspicture}
&
\begin{pspicture}(0,-.5)(5,.5)
\rput[l](4.05,.15){\scalebox{.75}{$\nu_6=7\nu_0$}}
\rput[l](4.05,-.15){\scalebox{.75}{$\lambda_6=2L/7$}}
\psaxes[yAxis=false,labels=none,Dx=4,ticksize=-.5 .5,linecolor=gray,linewidth=.02,linestyle=dashed](0,0)(4,0)
\psplot[plotstyle=curve,plotpoints=360]{0}{4}{315 x mul cos 2 div}
\psplot[plotstyle=curve,plotpoints=360]{0}{4}{0 315 x mul cos sub 2 div}
\end{pspicture}
&
\begin{pspicture}(0,-.5)(5,.5)
\rput[l](4.05,.15){\scalebox{.75}{$\nu_6=13\nu_0$}}
\rput[l](4.05,-.15){\scalebox{.75}{$\lambda_6=4L/13$}}
\psaxes[yAxis=false,labels=none,Dx=4,ticksize=-.5 .5,linecolor=gray,linewidth=.02,linestyle=dashed](0,0)(4,0)
\psplot[plotstyle=curve,plotpoints=360]{0}{4}{292.5 x mul sin 2 div}
\psplot[plotstyle=curve,plotpoints=360]{0}{4}{0 292.5 x mul sin sub 2 div}
\end{pspicture}
\\
\begin{pspicture}(0,-.5)(5,.5)
\rput[l](4.05,.15){\scalebox{.75}{$\nu_7=8\nu_0$}}
\rput[l](4.05,-.15){\scalebox{.75}{$\lambda_7=L/4$}}
\psaxes[yAxis=false,labels=none,Dx=4,ticksize=-.5 .5,linecolor=gray,linewidth=.02,linestyle=dashed](0,0)(4,0)
\psplot[plotstyle=curve,plotpoints=360]{0}{4}{360 x mul sin 2 div}
\psplot[plotstyle=curve,plotpoints=360]{0}{4}{0 360 x mul sin sub 2 div}
\end{pspicture}
&
\begin{pspicture}(0,-.5)(5,.5)
\rput[l](4.05,.15){\scalebox{.75}{$\nu_7=8\nu_0$}}
\rput[l](4.05,-.15){\scalebox{.75}{$\lambda_7=L/4$}}
\psaxes[yAxis=false,labels=none,Dx=4,ticksize=-.5 .5,linecolor=gray,linewidth=.02,linestyle=dashed](0,0)(4,0)
\psplot[plotstyle=curve,plotpoints=360]{0}{4}{360 x mul cos 2 div}
\psplot[plotstyle=curve,plotpoints=360]{0}{4}{0 360 x mul cos sub 2 div}
\end{pspicture}
&
\begin{pspicture}(0,-.5)(5,.5)
\rput[l](4.05,.15){\scalebox{.75}{$\nu_7=15\nu_0$}}
\rput[l](4.05,-.15){\scalebox{.75}{$\lambda_7=4L/15$}}
\psaxes[yAxis=false,labels=none,Dx=4,ticksize=-.5 .5,linecolor=gray,linewidth=.02,linestyle=dashed](0,0)(4,0)
\psplot[plotstyle=curve,plotpoints=360]{0}{4}{337.5 x mul sin 2 div}
\psplot[plotstyle=curve,plotpoints=360]{0}{4}{0 337.5 x mul sin sub 2 div}
\end{pspicture}
}
\end{tabular}}}
{\resizebox{.9\textwidth}{!}{\includegraphics{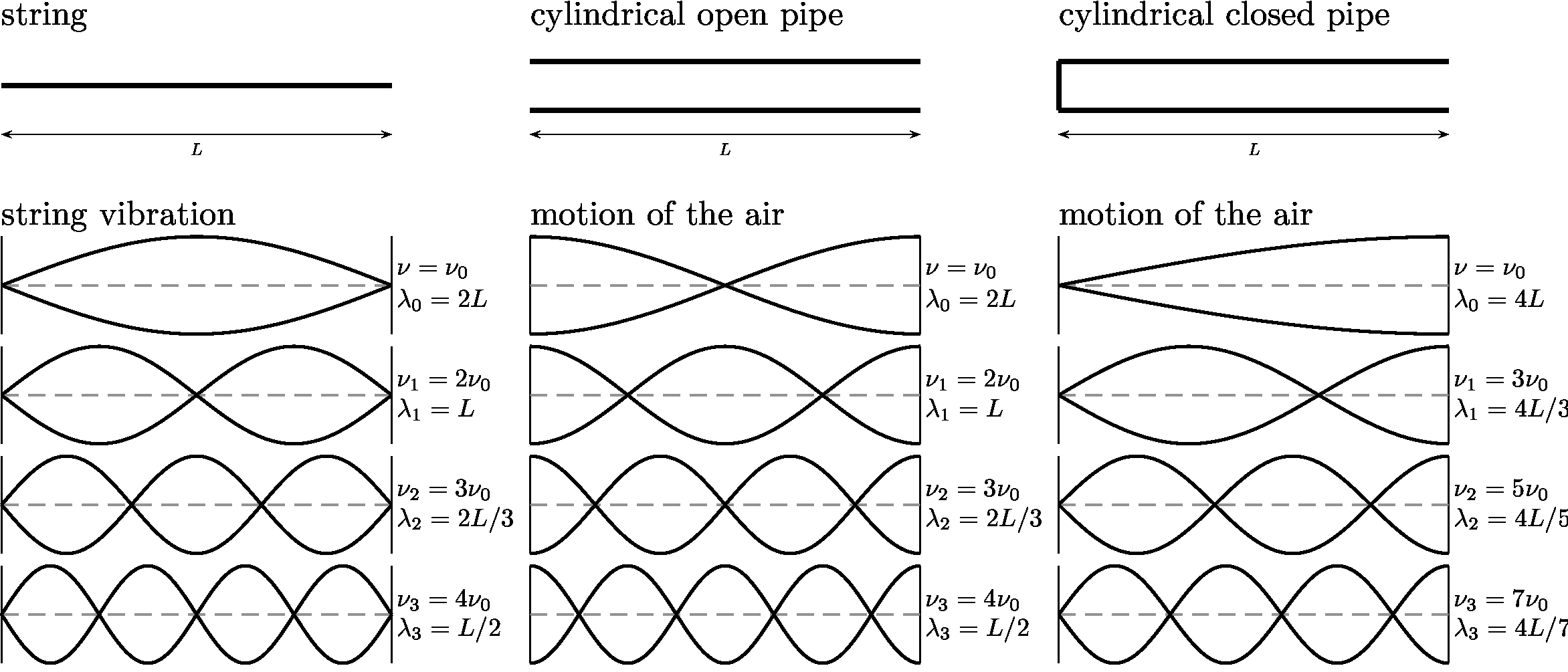}}}
\caption{Three models for standing waves and their harmonics: a string, a cylindrical open pipe, and cylindrical half-closed pipe.}
\label{f:waves}
\end{center}
\end{figure}
Figure~\ref{f:waves} illustrates three different models of standing waves. The first one corresponds to the vibration of a string with fixed ends,
the second one corresponds to the motion of the air inside a cylindrical pipe with two open ends such as a flute, and the third one corresponds to the motion of the air inside a cylindrical pipe with one open end and one closed end, such as a clarinet or an organ pipe.
In the three cases, the first wave represents a fundamental tone and the subsequent waves represent the harmonics associated with the fundamental tone.
In the example of the string and the open pipe, the frequencies of the harmonics are all integer multiples of the frequency of the fundamental, while in the example of a pipe with one open end and one closed end, the frequencies of the harmonics are all odd multiples of the frequency of the fundamental.
The octave corresponds to a ratio of frequencies
equal to $2$.

\subsection*{Equal temperaments} In twelve-tone equal temperament the octave is divided into $12$ equal semitones and the scale is tuned so that the ratio of the frequencies of consecutive semitones is $\sqrt[12]{2}$. 
%(see Figure~\ref{f:teclat}).
Thus, the ratio of the frequencies corresponding to a tone (i.e. two
semitones) is $(\sqrt[12]{2})^2$, the ratio of frequencies
corresponding to a minor third (a tone plus a semitone) is $(\sqrt[12]{2})^3$, and so on.
In particular, the ratio of the frequencies corresponding to an octave is $(\sqrt[12]{2})^{12}=2$.
Conversely, the number of semitones between notes of frequencies $f_1$ and $f_2$ is $\log_{\sqrt[12]{2}}(f_1/f_2)=12\log_2(f_1/f_2)$.
If another number, say $n$, of equal divisions of the octave is used,
then the frequencies of consecutive notes differ by a factor of
$\sqrt[n]{2}$, so that the number of $n$-th divisions
between notes with frequencies $f_1$ and $f_2$ is $n\log_{2}(f_1/f_2)$.

\subsection*{Matching the harmonic sequence and equal temperaments} The frequencies of any equal temperament do not match the frequencies of the pure notes in the harmonic series, since the frequencies of pure harmonic notes correspond to integer multiples of the frequency of the fundamental note while the frequencies in equal temperament correspond, except for the octaves, to nonrational multiples of the frequency of the fundamental. In this paper we deal with the properties of the harmonic series when matched to equal temperaments, that is, when the pure-harmonic tones are approximated by their neighboring counterparts in equal temperament.

\mut{
\newcommand\teclablanca[2]{\filldraw[thick,black,fill=black!#2](#1,2.5)[turtle={right=90,forward=1,right=90,forward=2.5,right=90,forward=1,right=90,forward=2.5}];}
\newcommand\teclanegra[2]{\filldraw[thick,black,fill=black!#2](#1,2.5)[turtle={right=90,forward=.8,right=90,forward=1.5,right=90,forward=.8,right=90,forward=1.5}];}

\begin{figure}
\begin{center}

\ifwithps{
\newcommand\keyboardI[8]{%
    \def\tempa{#1}%
    \def\tempb{#2}%
    \def\tempc{#3}%
    \def\tempd{#4}%
    \def\tempe{#5}%
    \def\tempf{#6}%
    \def\tempg{#7}%
    \def\temph{#8}%
}
\newcommand\keyboardII[5]{
\noindent\resizebox{.7\textwidth}{!}{\begin{tikzpicture}
\teclablanca{0}{0}
\teclablanca{1}{0}
\teclablanca{2}{0}
\teclablanca{3}{0}
\teclablanca{4}{0}
\teclablanca{5}{0}
\teclablanca{6}{0}
\teclablanca{7}{0}

\teclanegra{0.6}{20}
\teclanegra{1.6}{20}
\teclanegra{3.6}{20}
\teclanegra{4.6}{20}
\teclanegra{5.6}{20}

\node at (0.5,.5) {\tempa};
\node at (1.,1.5) {{\normalsize \tempb}};
\node at (1.5,.5) {\tempc};
\node at (2.,1.5) {{\normalsize \tempd}};
\node at (2.5,.5) {\tempe};
\node at (3.5,.5) {\tempf};
\node at (4.,1.5) {{\normalsize \tempg}};
\node at (4.5,.5) {\temph};
\node at (5.,1.5) {{\normalsize #1}};
\node at (5.5,.5) {#2};
\node at (6.,1.5) {{\normalsize #3}};
\node at (6.5,.5) {#4};
\node at (7.5,.5) {#5};
\end{tikzpicture}}}

\keyboardI{$\footnotesize{1}$}{$\footnotesize{{\sqrt[12\,]{2}}}$}{$\footnotesize{{\sqrt[12\,]{2}}^{\,2}}$}{$\footnotesize{{\sqrt[12\,]{2}}^{\,3}}$}{$\footnotesize{{\sqrt[12\,]{2}}^{\,4}}$}{$\footnotesize{{\sqrt[12\,]{2}}^{\,5}}$}{$\footnotesize{{\sqrt[12\,]{2}}^{\,6}}$}{$\footnotesize{{\sqrt[12\,]{2}}^{\,7}}$}
\keyboardII{$\footnotesize{{\sqrt[12\,]{2}}^{\,8}}$}{$\footnotesize{{\sqrt[12\,]{2}}^{\,9}}$}{$\footnotesize{{\sqrt[12\,]{2}}^{10}}$}{$\footnotesize{{\sqrt[12\,]{2}}^{\,11}}$}{$2$}
}{\resizebox{.7\textwidth}{!}{\includegraphics{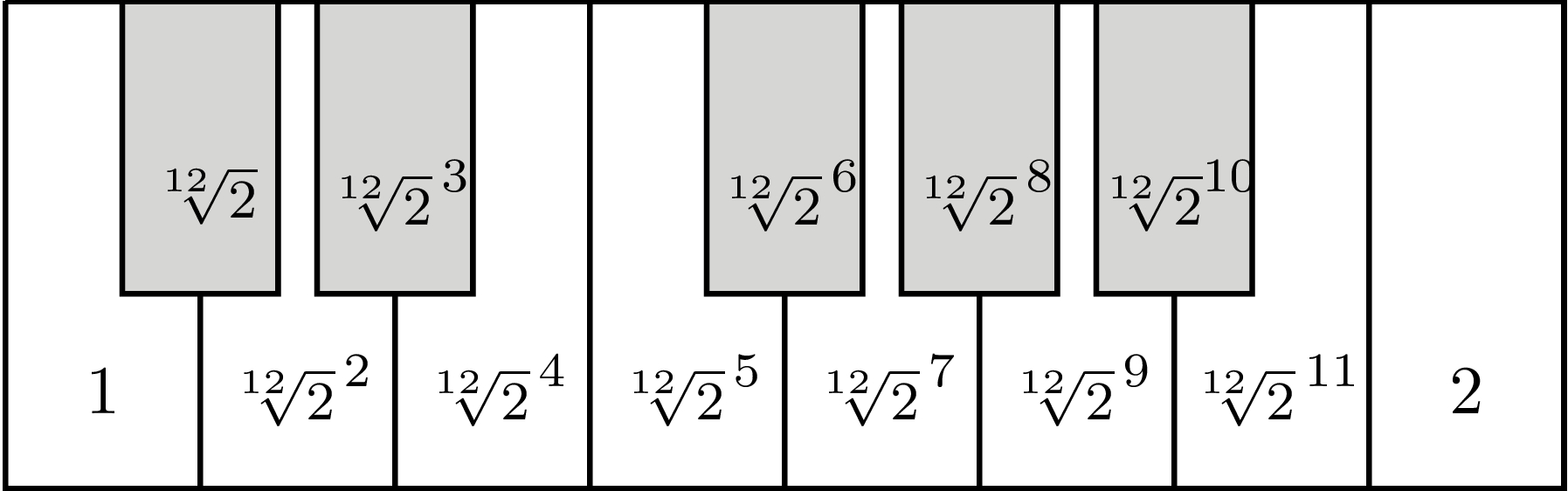}}}

\caption{Frequency ratios of scale notes with respect to the fundamental note in twelve-tone equal temperament.}  \label{f:teclat}
  \end{center}
\end{figure}
}

If one focuses on the example of the string with fixed ends, the notes corresponding to the harmonic series of C2 are approximately the ones in Figure~\ref{f:wodd} (a), when approximated by the twelve-tone equal temperament.
The same notes appear in cylindrical open pipes.
In the example of cylindrical pipes with one open end and one closed end, only half of these harmonics appear, namely, the ones in Figure~\ref{f:wodd} (b).

\begin{figure}
\begin{center}
\begin{tabular}{cc}
\resizebox{.55\textwidth}{!}{
\ifwithlilypond 
%{\lilypondfile[indent=2cm,quote,staffsize=12]{whiteharmonics.ly}}
{\lilypondfile[staffsize=16]{whiteharmonics.ly}}
{\includegraphics{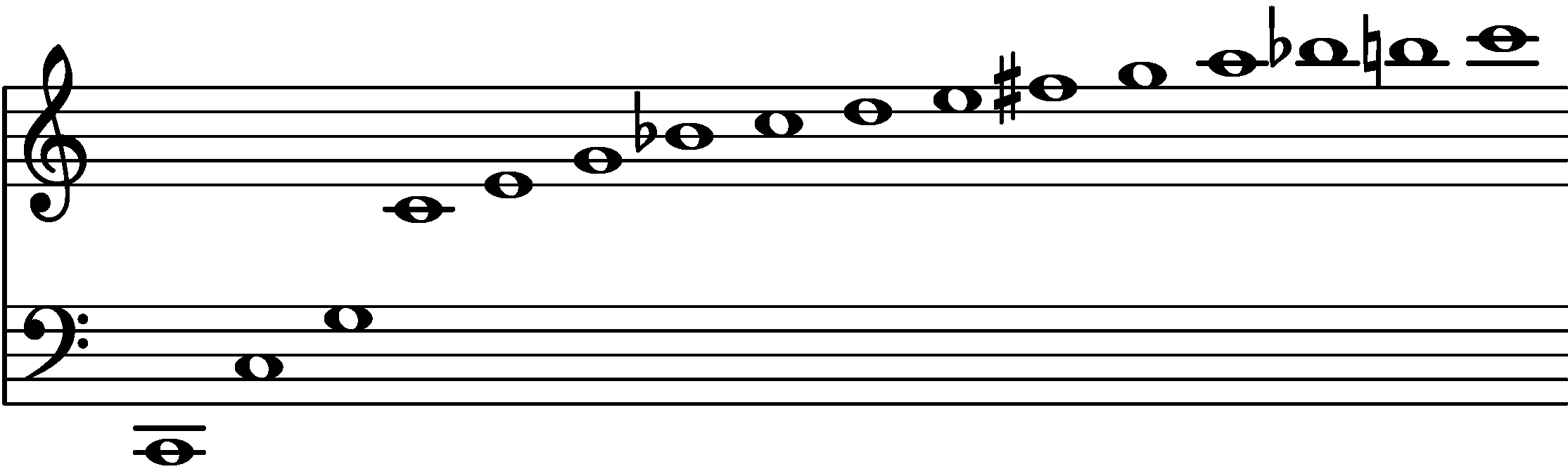}}}
&
%\ifwithlilypond{\lilypondfile[indent=3.75cm,quote,staffsize=12]{oddharmonics.ly}}
\resizebox{.35\textwidth}{!}{
\ifwithlilypond{\lilypondfile[staffsize=16]{oddharmonics.ly}}
{\includegraphics{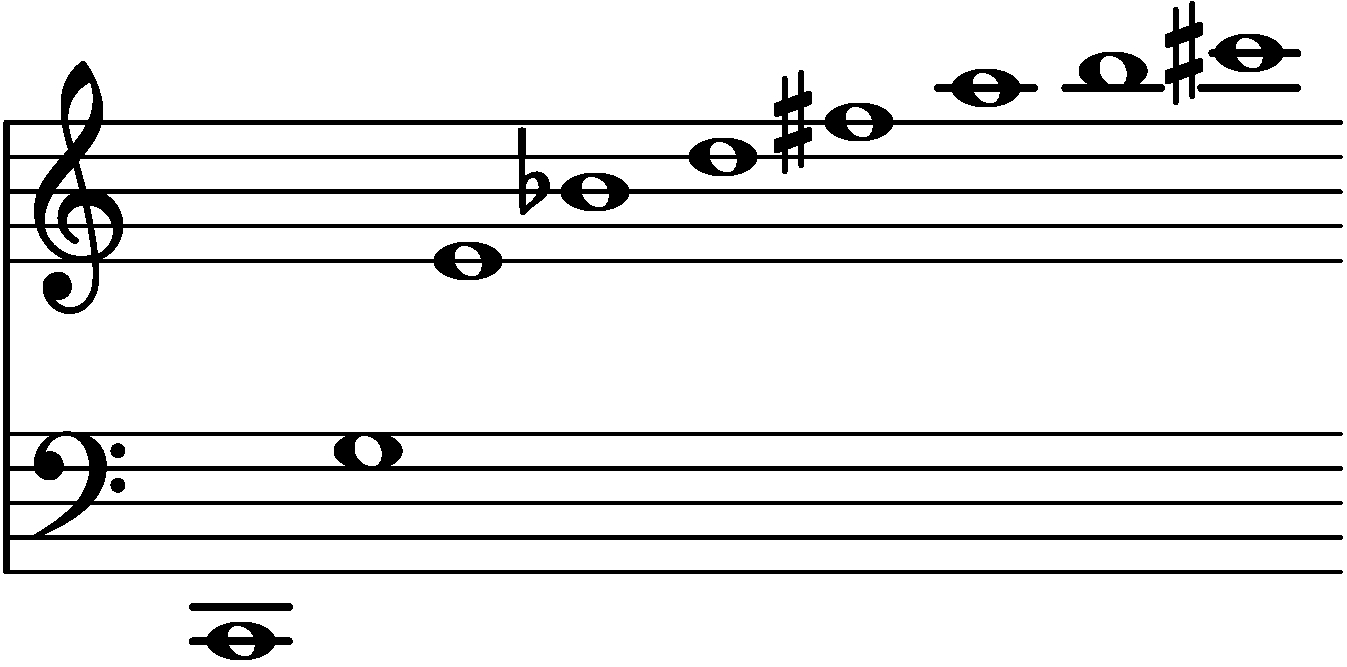}}}
\\
(a)&(b)
\end{tabular}
\end{center}

  \caption{(a) Approximation of the harmonics produced in a string with fixed ends, when translated to twelve-tone equal temperament. (b) Approximation of the harmonics produced in a pipe with one open end and one closed end, when translated to the twelve-tone equal temperament.}
\label{f:wodd}
\end{figure}

%Still in twelve-tone equal temperament, one can assign the value $0$ to C2, the value $1$ to C2\#, the value $2$ to D2, etc., the value $12$ to C3, the value $13$ to C3\# and so on. Then the sequence of harmonics of C2 produced by a string with fixed ends or by a cylindrical open pipe is represented by
For each note in the harmonic sequence of Figure~\ref{f:wodd} (a), take its distance in semitones to the fundamental tone. We get the following set.
\begin{equation}\label{eq:H}
  H=\{0,     12,                  19,
     24,         28,      31,      34,
     36,   38,   40,   42,43\}\cup\{i\in{\mathbb N}: i\geq 45\}.
\end{equation}
These semitone intervals collapse at some point, since there is not a one-to-one correspondence from pure harmonics to semitone intervals.
%Notice that one could have used any other pitch rather than C2.
%, and any other number of divisions of the octave rather than $12$, but C2 and $12$ were used to make the explanation more familiar and more simple.

\subsection*{The discrete model of numerical semigroups}
The set $H$ is the result of amplifying the set of logarithms $\{\log_2{1},\log_2{2},\log_2{3},\dots\}$ by a factor of 12 and then rounding the obtained real numbers to integer numbers.
The logarithms, as is vastly known, and as will be expained later, appear when one thinks of the values of notes as their relative interval with respect to the fundamental and when one requires the intervals between notes to behave like a proper distance. This is related to what we call the {\em product-compatibility} property.
Notice that $12\log_2(13)=44.4$ approximately and so, depending on the rounding criterion, the $13$th harmonic can be considered to be at $44$ semitones or at $45$ semitones of the fundamental.
We chose to consider the 13th harmonic at 45 semitones
of fundamental tone, as it has been preferred by traditional authors. This will be justified later. More recently, though, composers influenced by ``spectral music'' have often represented this harmonic by the pitch 44 semitones above the fundamental. See \citet{Tymoczko} for another mathematical justification of representing the 13th harmonic as 45 semitones above the fundamental.

The procedure of amplifying a sequence of real numbers and then mapping them to integer numbers is what we call a {\em discretization} of the initial sequence of real numbers. This way, $H$ is a discretization of the sequence of logarithms with amplifying factor $12$ and rounding threshold $0.4$.

The sequence of logarithm-related real numbers is an ideal physical model while its discretization is a feasible realization of it. Natural properties of the discretization, in the context of harmonics, are:  (i) it is a subset of non-negative integers containing $0$, (ii) only finitely many non-negative integers are missing, and (iii) it is closed under addition.
The addition closure amounts to the fact that the harmonics of a harmonic of a fundamental tone should be harmonics of that fundamental tone. These three properties are exactly the properties defining a {\em numerical semigroup} (see Section~\ref{s:ns}). Hence, we call the set $H$ the {\em well-tempered harmonic semigroup}.

Discretizing the sequence of physical harmonics into numerical semigroups may be considered tantamount to setting equal temperaments. The number of equal parts of the octave in an equal temperament is the amplifying factor in the discretization and it corresponds to the {\em multiplicity} of the related numerical semigroup, i.e. its smallest non-zero element. The present paper treats special ideal sequences of real numbers and their discretizations into numerical semigroups, using a variety of multiplicities to discretize, so obtaining different numerical semigroups.

\subsection*{The ${\mathbb R}$ model of tempered monoids}

The ideal sets of real numbers that we want to discretize must be 
an increasing sequence of non-negative real numbers.
Other natural properties of these sequences, paralleling what has been
said in the previous paragraphs, are that (i) they must contain $0$,
(ii) the limit of the differences of consecutive elements is zero, 
and (iii) the sequences must be closed under addition.
We say that an increasing sequence satisfying (i), (ii), and (iii) is a {\em tempered monoid} (see Section~\ref{s:ns}).

Let us concentrate on two additional properties that tempered monoids can have, which have a translation into the harmonic series. The first one is the {\em product-compatibility}, arising when one tries to fit the multiplicative nature of harmonics and frequency ratios in the additive environment of pitch and interval distances.
Indeed, suppose that the pitch difference of each harmonic with respect to the fundamental tone is represented by the increasing sequence $\rho_1=0,\rho_2,\rho_3,\dots$.
Now, the difference between the pitch of the third harmonic and the pitch of the fundamental tone must be the same as the difference between the pitch of the third harmonic of the fifth harmonic and the pitch of the fifth harmonic itself (see Figure~\ref{f:wpyth}).
This means that $\rho_{15}-\rho_5=\rho_3$. In general, for any positive integers $i$, $j$, it is required that $\rho_{ij}-\rho_{i}=\rho_{j}$ or, equivalently, $\rho_{ij}=\rho_{i}+\rho_{j}$. The tempered monoids satisfying this equality for any positive integers $i,j$ will be called product-compatible tempered monoids.

\begin{figure}[h]
  \begin{center}
\resizebox{.65\textwidth}{!}{\ifwithlilypond{
    \lilypondfile[staffsize=16]{whiteharmonicspyth.ly}
}{\includegraphics{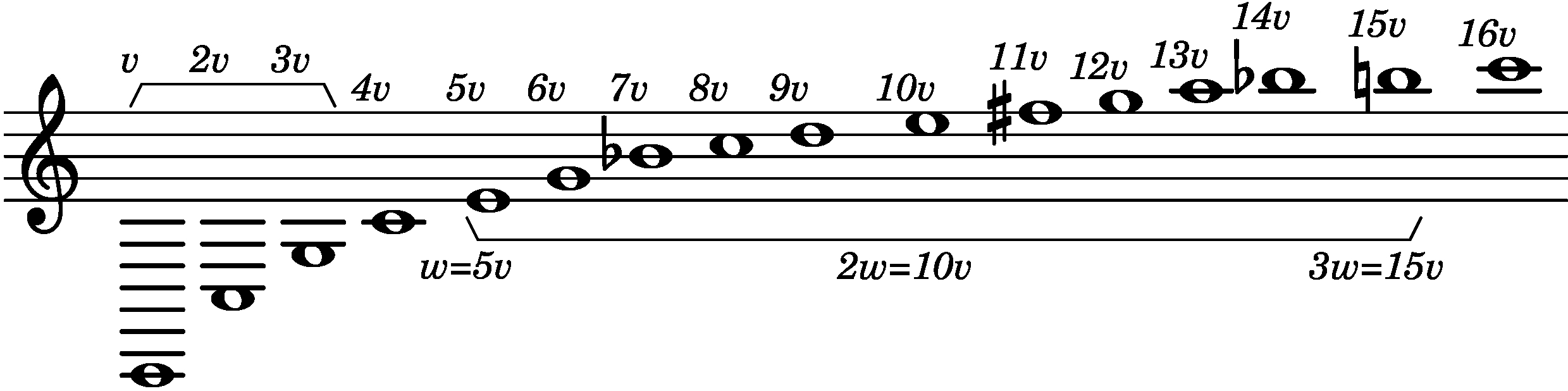}}}
  \caption{The difference between the pitch of the third harmonic and the pitch of the fundamental, as marked above, must be the same as the difference between the pitch of the third harmonic of the fifth harmonic and the pitch of the fifth harmonic itself, as marked below.
  }
\label{f:wpyth}
\end{center}
\end{figure}

The second additional property of interest of a tempered monoid is that of being {\em fractal}.
We can divide a segment in a fractal way as follows. First
we halve it. We call this a bisection of the interval since the two parts in which we divide the interval are equal. Then we halve each half and so on, indefinitely (see Figure~\ref{f:fractaldivisionofaninterval} (a)).
The same idea can be applied by dividing the interval into two parts in
a given proportion, not necessarily into two equal parts, and so, not necessarily a bisection. Next, divide each
of the parts following the same proportions as in the first cut. Divide
again each of the parts in the same proportions and so on. We obtain an
apparently chaotic but strictly fractal partition (see 
Figure~\ref{f:fractaldivisionofaninterval} (b)).

\noindent
\begin{figure}
\noindent
\begin{tabular}{ccc}
\begin{minipage}{.375\textwidth}
  \centering
\ifwithps{
  \psset{unit=5.25cm}
\begin{tabular}{c}
\begin{pspicture}(0,0)(1,.125)
\psaxes[labels=none,ticks=none,linewidth=.005](0,0.1)(1,0.1)
\psaxes[labels=none,ticks=none,linewidth=.005](0,0.075)(0,0.125)
\put(0,.05){\makebox(0,0){$0$}}
\psaxes[labels=none,ticks=none,linewidth=.005](1,0.075)(1,0.125)
\put(1,.05){\makebox(0,0){$1$}}
\end{pspicture}
\\
\begin{pspicture}(0,0)(1,.125)
\psaxes[labels=none,ticks=none,linewidth=.005](0,0.1)(1,0.1)
\psaxes[labels=none,ticks=none,linewidth=.005](0,0.075)(0,0.125)
\put(0,.05){\makebox(0,0){$0$}}
\psaxes[labels=none,ticks=none,linewidth=.005](1,0.075)(1,0.125)
\put(1,.05){\makebox(0,0){$1$}}
\psaxes[labels=none,ticks=none,linewidth=.005](.5,0.075)(.5,0.125)
\put(.5,.05){\makebox(0,0){$.5$}}
\end{pspicture}
\\
\begin{pspicture}(0,0)(1,.125)
\psaxes[labels=none,ticks=none,linewidth=.005](0,0.1)(1,0.1)
\psaxes[labels=none,ticks=none,linewidth=.005](0,0.075)(0,0.125)
\put(0,.05){\makebox(0,0){$0$}}
\psaxes[labels=none,ticks=none,linewidth=.005](1,0.075)(1,0.125)
\put(1,.05){\makebox(0,0){$1$}}
\psaxes[labels=none,ticks=none,linewidth=.005](.5,0.075)(.5,0.125)
\put(.5,.05){\makebox(0,0){$.5$}}
\psaxes[labels=none,ticks=none,linewidth=.005](.25,0.075)(.25,0.125)
\put(.25,.05){\makebox(0,0){$.25$}}
\psaxes[labels=none,ticks=none,linewidth=.005](.75,0.075)(.75,0.125)
\put(.75,.05){\makebox(0,0){$.75$}}
\end{pspicture}
\\
\begin{pspicture}(0,0)(1,.125)
\psaxes[labels=none,ticks=none,linewidth=.005](0,0.1)(1,0.1)
\psaxes[labels=none,ticks=none,linewidth=.005](0,0.075)(0,0.125)
\put(0,.05){\makebox(0,0){$0$}}
\psaxes[labels=none,ticks=none,linewidth=.005](1,0.075)(1,0.125)
\put(1,.05){\makebox(0,0){$1$}}
\psaxes[labels=none,ticks=none,linewidth=.005](.5,0.075)(.5,0.125)
\put(.5,.05){\makebox(0,0){$.5$}}
\psaxes[labels=none,ticks=none,linewidth=.005](.25,0.075)(.25,0.125)
\put(.25,.05){\makebox(0,0){$.25$}}
\psaxes[labels=none,ticks=none,linewidth=.005](.75,0.075)(.75,0.125)
\put(.75,.05){\makebox(0,0){$.75$}}
\psaxes[labels=none,ticks=none,linewidth=.005](.125,0.075)(.125,0.125)
\put(.125,.05){\makebox(0,0){$.125$}}
\psaxes[labels=none,ticks=none,linewidth=.005](.375,0.075)(.375,0.125)
\put(.375,.05){\makebox(0,0){$.375$}}
\psaxes[labels=none,ticks=none,linewidth=.005](.625,0.075)(.625,0.125)
\put(.625,.05){\makebox(0,0){$.625$}}
\psaxes[labels=none,ticks=none,linewidth=.005](.875,0.075)(.875,0.125)
\put(.875,.05){\makebox(0,0){$.875$}}
\end{pspicture}
\\
\begin{pspicture}(0,0)(1,.125)
\psaxes[labels=none,ticks=none,linewidth=.005](0,0.1)(1,0.1)
\psaxes[labels=none,ticks=none,linewidth=.005](0,0.075)(0,0.125)
\put(0,.05){\makebox(0,0){$0$}}
\psaxes[labels=none,ticks=none,linewidth=.005](1,0.075)(1,0.125)
\put(1,.05){\makebox(0,0){$1$}}
\psaxes[labels=none,ticks=none,linewidth=.005](.5,0.075)(.5,0.125)
\psaxes[labels=none,ticks=none,linewidth=.005](.25,0.075)(.25,0.125)
\psaxes[labels=none,ticks=none,linewidth=.005](.75,0.075)(.75,0.125)
\psaxes[labels=none,ticks=none,linewidth=.005](.125,0.075)(.125,0.125)
\psaxes[labels=none,ticks=none,linewidth=.005](.375,0.075)(.375,0.125)
\psaxes[labels=none,ticks=none,linewidth=.005](.625,0.075)(.625,0.125)
\psaxes[labels=none,ticks=none,linewidth=.005](.875,0.075)(.875,0.125)
\psaxes[labels=none,ticks=none,linewidth=.005](.0625,0.075)(.0625,0.125)
\psaxes[labels=none,ticks=none,linewidth=.005](.1875,0.075)(.1875,0.125)
\psaxes[labels=none,ticks=none,linewidth=.005](.3125,0.075)(.3125,0.125)
\psaxes[labels=none,ticks=none,linewidth=.005](.4375,0.075)(.4375,0.125)
\psaxes[labels=none,ticks=none,linewidth=.005](.5625,0.075)(.5625,0.125)
\psaxes[labels=none,ticks=none,linewidth=.005](.6875,0.075)(.6875,0.125)
\psaxes[labels=none,ticks=none,linewidth=.005](.8125,0.075)(.8125,0.125)
\psaxes[labels=none,ticks=none,linewidth=.005](.9375,0.075)(.9375,0.125)
\end{pspicture}
\end{tabular}}{\resizebox{.9\textwidth}{!}{\includegraphics{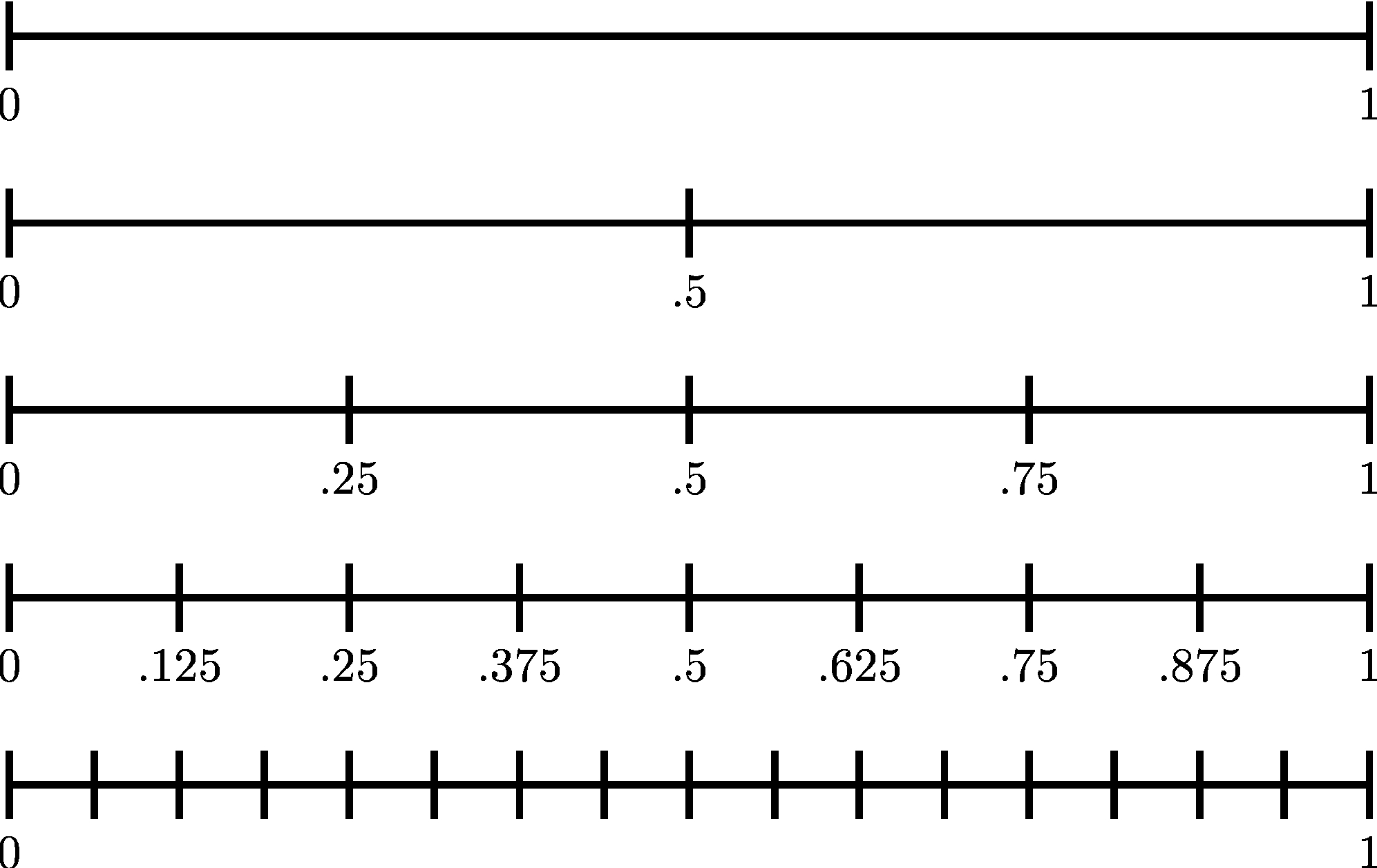}}}
\end{minipage}
&\ \ \ \ \ &
\begin{minipage}{.375\textwidth}
  \centering
\ifwithps{
\psset{unit=5.25cm}
\begin{tabular}{c}
\begin{pspicture}(0,0)(1,.125)
\psaxes[labels=none,ticks=none,linewidth=.005](0,0.1)(1,0.1)
\psaxes[labels=none,ticks=none,linewidth=.005](0,0.075)(0,0.125)
\put(0,.05){\makebox(0,0){$0$}}
\psaxes[labels=none,ticks=none,linewidth=.005](1,0.075)(1,0.125)
\put(1,.05){\makebox(0,0){$1$}}
\end{pspicture}
\\
\begin{pspicture}(0,0)(1,.125)
\psaxes[labels=none,ticks=none,linewidth=.005](0,0.1)(1,0.1)
\psaxes[labels=none,ticks=none,linewidth=.005](0,0.075)(0,0.125)
\put(0,.05){\makebox(0,0){$0$}}
\psaxes[labels=none,ticks=none,linewidth=.005](1,0.075)(1,0.125)
\put(1,.05){\makebox(0,0){$1$}}
\psaxes[labels=none,ticks=none,linewidth=.005](.618,0.075)(.618,0.125)
\put(.618,.05){\makebox(0,0){$p$}}
\end{pspicture}
\\
\begin{pspicture}(0,0)(1,.125)
\psaxes[labels=none,ticks=none,linewidth=.005](0,0.1)(1,0.1)
\psaxes[labels=none,ticks=none,linewidth=.005](0,0.075)(0,0.125)
\put(0,.05){\makebox(0,0){$0$}}
\psaxes[labels=none,ticks=none,linewidth=.005](1,0.075)(1,0.125)
\put(1,.05){\makebox(0,0){$1$}}
\psaxes[labels=none,ticks=none,linewidth=.005](.618,0.075)(.618,0.125)
\put(.618,.05){\makebox(0,0){$p$}}
\psaxes[labels=none,ticks=none,linewidth=.005](.382,0.075)(.382,0.125)
\put(.382,.05){\makebox(0,0){$p^2$}}
\psaxes[labels=none,ticks=none,linewidth=.005](.8541,0.075)(.8541,0.125)
\put(.8541,.05){\makebox(0,0){\scalebox{.8}{$p+(1-p)p$}}}
\end{pspicture}
\\
\begin{pspicture}(0,0)(1,.125)
\psaxes[labels=none,ticks=none,linewidth=.005](0,0.1)(1,0.1)
\psaxes[labels=none,ticks=none,linewidth=.005](0,0.075)(0,0.125)
\put(0,.05){\makebox(0,0){$0$}}
\psaxes[labels=none,ticks=none,linewidth=.005](1,0.075)(1,0.125)
\put(1,.05){\makebox(0,0){$1$}}
\psaxes[labels=none,ticks=none,linewidth=.005](.618,0.075)(.618,0.125)
\psaxes[labels=none,ticks=none,linewidth=.005](.382,0.075)(.382,0.125)
\psaxes[labels=none,ticks=none,linewidth=.005](.8541,0.075)(.8541,0.125)
\psaxes[labels=none,ticks=none,linewidth=.005](.2361,0.075)(.2361,0.125)
\psaxes[labels=none,ticks=none,linewidth=.005](.5279,0.075)(.5279,0.125)
\psaxes[labels=none,ticks=none,linewidth=.005](.7639,0.075)(.7639,0.125)
\psaxes[labels=none,ticks=none,linewidth=.005](.9443,0.075)(.9443,0.125)
\end{pspicture}
\\
\begin{pspicture}(0,0)(1,.125)
\psaxes[labels=none,ticks=none,linewidth=.005](0,0.1)(1,0.1)
\psaxes[labels=none,ticks=none,linewidth=.005](0,0.075)(0,0.125)
\put(0,.05){\makebox(0,0){$0$}}
\psaxes[labels=none,ticks=none,linewidth=.005](1,0.075)(1,0.125)
\put(1,.05){\makebox(0,0){$1$}}
\psaxes[labels=none,ticks=none,linewidth=.005](.618,0.075)(.618,0.125)
\psaxes[labels=none,ticks=none,linewidth=.005](.382,0.075)(.382,0.125)
\psaxes[labels=none,ticks=none,linewidth=.005](.8541,0.075)(.8541,0.125)
\psaxes[labels=none,ticks=none,linewidth=.005](.2361,0.075)(.2361,0.125)
\psaxes[labels=none,ticks=none,linewidth=.005](.5279,0.075)(.5279,0.125)
\psaxes[labels=none,ticks=none,linewidth=.005](.7639,0.075)(.7639,0.125)
\psaxes[labels=none,ticks=none,linewidth=.005](.9443,0.075)(.9443,0.125)
\psaxes[labels=none,ticks=none,linewidth=.005](.1459 ,0.075)(.1459,0.125)
\psaxes[labels=none,ticks=none,linewidth=.005](.3262 ,0.075)(.3262,0.125)
\psaxes[labels=none,ticks=none,linewidth=.005](.4721 ,0.075)(.4721,0.125)
\psaxes[labels=none,ticks=none,linewidth=.005](.5836,0.075)(.5836,0.125)
\psaxes[labels=none,ticks=none,linewidth=.005](.7082,0.075)(.7082,0.125)
\psaxes[labels=none,ticks=none,linewidth=.005](.8197,0.075)(.8197,0.125)
\psaxes[labels=none,ticks=none,linewidth=.005](.9098,0.075)(.9098,0.125)
\psaxes[labels=none,ticks=none,linewidth=.005](.9787,0.075)(.9787,0.125)
\end{pspicture}
\end{tabular}}{\resizebox{.9\textwidth}{!}{\includegraphics{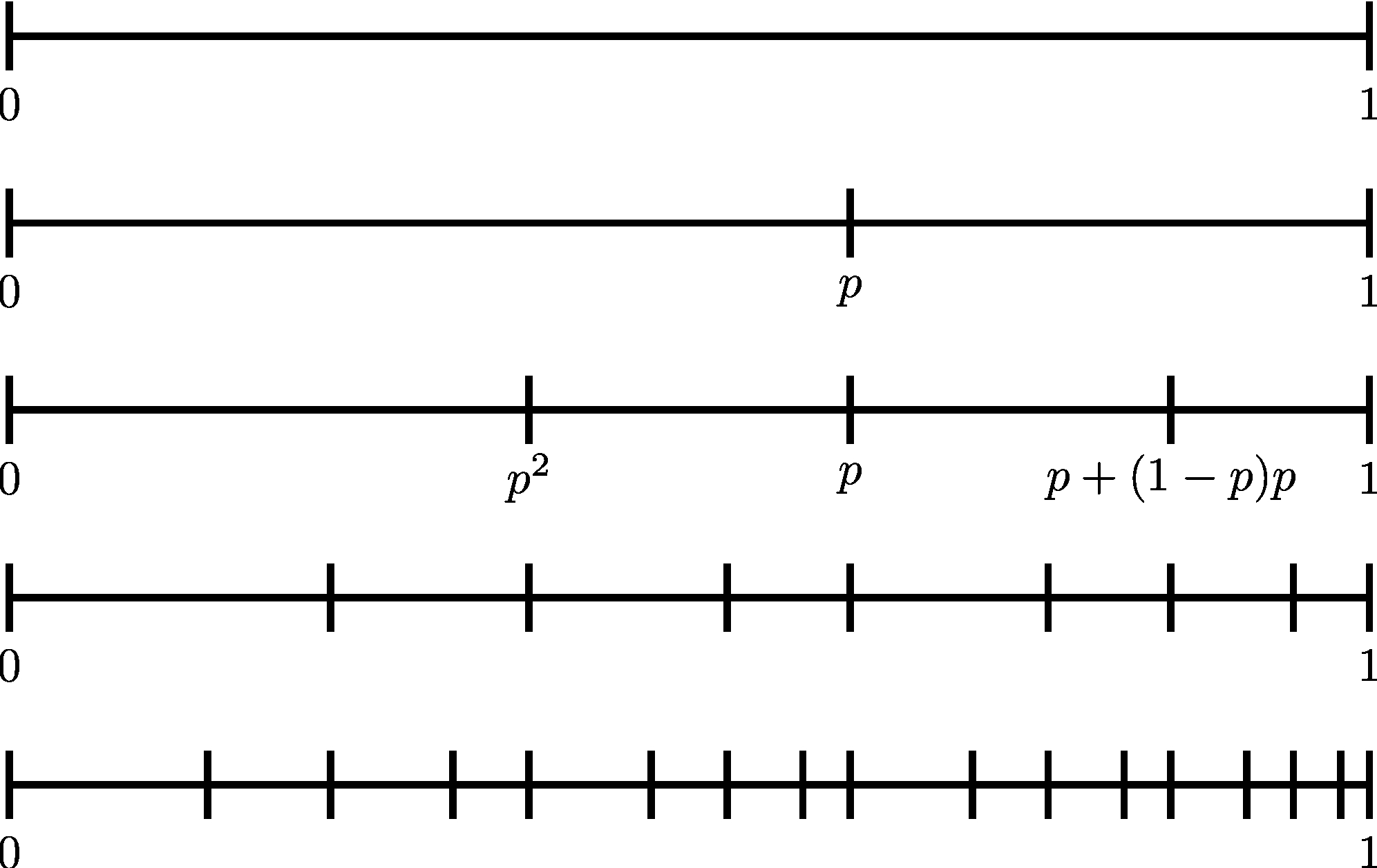}}}
\end{minipage}
\\
(a)&&(b)
\end{tabular}
\caption{(a) Fractal division of an interval (bisectional). (b) Fractal division of an interval (nonbisectional).}
\label{f:fractaldivisionofaninterval}
\end{figure}

Similarly, take the octave as the segment to be divided and restrict to the twelve-tone equal temperament approximation of harmonics. We can observe that in each next octave in the harmonic series, the intervals appearing in the previous octave are divided always into two parts with the same portions (see Figure~\ref{f:l}).
Observe that the octave between the second and fourth harmonics is divided by the third harmonic leaving $7/12$ of the octave on the left and $5/12$ on the right. The two intervals obtained now, C-G, G-C, are then repeated within the octave comprised between the fourth and eighth harmonics and subsequently divided leaving the best possible approximation (restricted to the twelve-tone equal temperament) of $7/12$ of the interval on the left and the best possible approximation of $5/12$ of the interval on the right, thus obtaining C-E, E-G and G-B$\flat$, B$\flat$-C. Each black note in Figure~\ref{f:l} is a new note that did not appear in any previous octave, and which divides an interval of two notes appearing in the previous octave in the same portions of approximately $7/12$ of the interval on the left and $5/12$ of the interval on the right.
The division of the octave into $12$ equal parts is used here in the introduction for simplicity. In the following sections all possible divisions of the octave are considered {\em a priori}.

\begin{figure}[h]
\begin{center}
\resizebox{.6\textwidth}{!}{
\ifwithlilypond{
  \lilypondfile[staffsize=16]{fractaldivisionharmonics.ly}
  }{\includegraphics{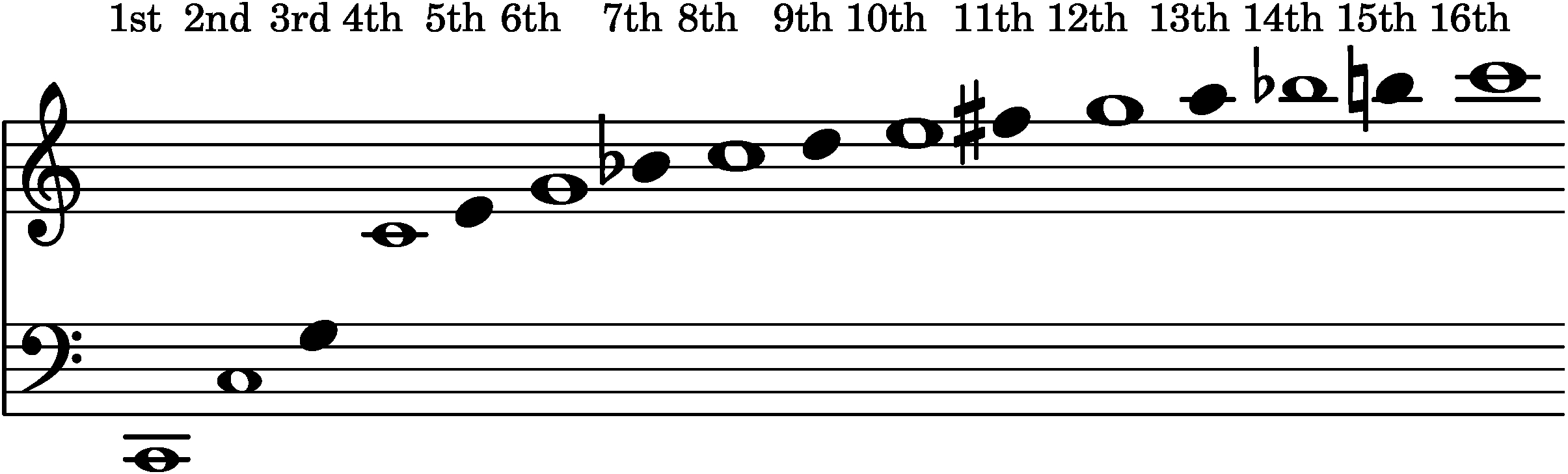}}}
  \caption{
    Each black note in the figure is a new note that did not appear in any previous octave, and which divides an interval of two notes appearing in the previous octave in the same portions of approximately $7/12$ of the interval on the left and $5/12$ of the interval on the right.}
\label{f:l}
\end{center}
\end{figure}

Lastly, the frequencies of the harmonics of cylindrical pipes with one open end and one closed end suggest to analyze the so-called {\em odd-filterable} tempered monoids and {\em odd-filterable} numerical semigroups. For tempered monoids, {\em odd-filterable} means that the subsequence of elements with odd index is closed under addition. The discretization of a tempered monoid into a numerical semigroup collapses at some point, in the sense that two different real elements of the tempered monoid map to the same integer element of the numerical semigroup. Now, we say that a numerical semigroup is odd-filterable if the subsequence of elements with odd index is closed under addition before collapse. That is, the sums of pairs in it that are smaller than the collapse are either in the subsequence or are larger than the collapse.

\subsection*{Our results}
In this paper we prove that there is only one product-compatible tempered monoid up to normalization, and that there is only one nonbisectional fractal monoid, up to normalization. Furthermore we will prove that the unique nonbisectional division for fractal monoids is nothing else but the inverse of the golden ratio. The quotient $7/12$ used in the previous paragraph is an approximation of it. Notice that although there are much better rational approximations to the golden ratio, $7/12$ is its closest fraction when restricted to denominator $12$.

Then we will show that the discretization of the unique product-compatible tempered monoid and the discretization of the unique nonbisectional fractal monoid coincide only when we divide the octave into $1$, $2$, $3$, $4$, $5$, $6$, $7$, $8$, $9$, $10$, $12$, $13$, or $18$ parts.
We finally will prove that when dividing the octave into $13$ or $18$ equal parts, the related numerical semigroups are not odd-filterable, but when dividing it into $12$ equal parts, it keeps the odd-filterability property.

This allows us to conclude that the maximum number of equal divisions of the octave such that the discretization of the unique product-compatible tempered monoid is simultaneously fractal and odd-filterable is $12$. The resulting numerical semigroup is the well-tempered harmonic semigroup $H$ in \eqref{eq:H}. This gives an alternative justification for the choice of number $12$ in the division of the octave into equal parts.

For a mathematical explanation of different stationary waves produced in different instruments and temperaments see \citet{Hall}, \citet{Harkleroad}, and \citet{Sethares}.
Other references treating the number of equal divisions of the octave are \citet{Barbour}, \citet{Douthett}, and \citet{Krantz}. For a general reference on numerical semigroups see \citet{Pedro}.

\section{Numerical semigroups and tempered monoids}
\label{s:ns}

In this section we introduce the two main mathematical objects of this paper. They are {\em numerical semigroups} and the newly defined {\em tempered monoids}.
Both are examples of increasingly enumerable submonoids of $({\mathbb R}^+,+)$, called $\omega$-monoids in \cite{omegamonoids}. Different results related to submonoids of $({\mathbb R}^+,+)$ can be found in \cite{Clifford,Grams,Anderson,Gotti,GottiGotti}.

\begin{definition}
A {\em numerical semigroup} is 
an additive submonoid of ${\mathbb N}_0$ with finite complement in ${\mathbb N}_0$.
%subset $S$ of ${\mathbb N_0}:={\mathbb N}\cup\{0\}$ such that 
%\begin{enumerate}
%\item $S$ contains $0$,
%\item $S$ has finite complement in ${\mathbb N}_0$,
%\item $S$ is closed under addition, that is, for any pair $a,b\in S$, the sum $a+b$ belongs to $S$.
%\end{enumerate}
%The elements in the complement are called the {\em gaps} of $S$ and the number of gaps is the {\em genus} of $S$, denoted $g(S)$.
The second element of $S$, i.e., the smallest non-zero element of $S$, is its {\em multiplicity}, denoted $m(S)$.
\end{definition}
The set $H$ in equation \eqref{eq:H} is a numerical semigroup of multiplicity $12$.

%\begin{example}
%  \label{e:hermite}
%  The set $$S=\{0,4,5,8,9,10\}\cup\{i\in{\mathbb N}: i\geq 12\}$$ is a numerical semigroup.
%  %  Its gaps are $1,2,3,6,7,11$ and its genus and multiplicity are $g(S)=6$ and $m(S)=4$.
%  Its gaps are $1,2,3,6,7,11$ and its multiplicity is $m(S)=4$.
%\end{example}

In this paper we introduce the notion of {\em tempered monoid}.
\begin{definition}
A {\em tempered monoid} 
is a strictly increasing sequence
$M=\{\mu_1,\mu_2,\mu_3,\dots\}$ of non-negative
real numbers such that (i) $\mu_1=0$, (ii) $\lim_{n\rightarrow\infty}(\mu_{n+1}-\mu_{n})=0$, (iii) 
%for any $\varepsilon\in{\mathbb R}$ with $\varepsilon>0$,  there exists some $n_0\in {\mathbb N}$ such that for any $n\in {\mathbb N}, n\geq n_0$, we have  $\mu_{n+1}-\mu_n<\varepsilon$,
$M$ is closed under addition, that is, for any $i,j\in{\mathbb N}$, there exists $k\in{\mathbb N}$ such that $\mu_k=\mu_i+\mu_j$.

A tempered monoid is called {\em normalized} if its smallest non-zero element is $1$.
If a tempered monoid $M$ is normalized, for $i>0$ we define its {\em $i$-th period}, denoted $\pi_i(M)$, as the set of elements in $M$ that are at least as large as $i$ and that are strictly smaller than $i+1$.
The {\em granularity} of $M$ is the cardinality of its first period.
\end{definition}

\begin{example}
  \label{e:quarterstempmonoid}
    The following set is a tempered monoid.
$$Q=\{0\}\cup\{n+ \frac{k}{2^{n+1}} :n\in {\mathbb N} \mbox{ and } 0\leq k\leq 2^{n+1}-1\}.$$
Its granularity is $4$. Its first period is $$\{1, 1+\frac{1}{4},1+\frac{1}{2},1+\frac{3}{4}\}.$$
%Its second period is $\{2,2+\frac{1}{8},2+\frac{2}{8},2+\frac{3}{8},2+\frac{4}{8},2+\frac{5}{8},2+\frac{6}{8},2+\frac{7}{8}\}$. Its third period is $\{3,3+\frac{1}{16},3+\frac{2}{16},3+\frac{3}{16},3+\frac{4}{16},3+\frac{5}{16},3+\frac{6}{16},3+\frac{7}{16},3+\frac{8}{16},3+\frac{9}{16},3+\frac{10}{16},3+\frac{11}{16},3+\frac{12}{16},3+\frac{13}{16},3+\frac{14}{16},3+\frac{15}{16}\}$, and so on.
\end{example}

\begin{example}
\label{e:deustempmonoid}
  The following set is a tempered monoid.
$$D=\{0\}\cup\{n+ \frac{k}{10^{n}} :n\in {\mathbb N} \mbox{ and } 0\leq k\leq 10^{n}-1\}$$
\mut{  \begin{eqnarray*}
  D&=&0,\\
  &&1,1.1,1.2,1.3,1.4,1.5,1.6,1.7,1.8,1.9,\\
  &&2,2.01,2.02,2.03,2.04,\dots
  2.94,2.95,2.96,2.97,2.98,2.99,\\
  &&3,3.001,3.002,3.003,3.004,\dots,3.994,3.995,3.996,3.997,3.998,3.999,\\
  &&4,4.0001,4.0002,4.0003,4.0004,\dots,4.9994,4.9995,4.9996,4.9997,4.9998,4.9999,\\
  &&5,5.00001,5.00002,5.00003,5.00004,\dots
\end{eqnarray*}}
  Its granularity is $10$ and its $i$-th period has cardinality $10^i$.
\end{example}

\section{Discretization of tempered monoids to obtain numerical semigroups}

In this section we explain how we can obtain numerical semigroups by discretizing the elements of tempered monoids.

For a real number $r$ let  $[r]$ be its {\em integer rounding} by the nearest integer, let $\lfloor r \rfloor$ be its {\em integer floor}, and let $\lceil r \rceil$ be its {\em integer ceiling}.
More generally, given $r\in{\mathbb R}$ and $\alpha\in{\mathbb R}$ with $0\leq \alpha\leq 1$, let $$\lfloor r\rfloor_\alpha =\left\{\begin{array}{ll}
\lfloor r\rfloor&\mbox{ if }r-\lfloor r\rfloor < \alpha,\\
\lceil r\rceil&otherwise
\end{array}\right.$$
In particular, $\lfloor\cdot\rfloor=\lfloor\cdot\rfloor_1$,
$\lceil\cdot\rceil=\lfloor\cdot\rfloor_0$,
$[\cdot]=\lfloor\cdot\rfloor_{0.5}$.

For a tempered monoid $M$ and a positive integer $m$ we can similarly apply 
the operations to the tempered monoid elementwise:
$$[mM]:=\{[mr]:r\in M\},\mbox{ \ \ }\lfloor mM\rfloor:=\{\lfloor mr\rfloor:r\in M\},\mbox{ \ \ }\lceil mM\rceil:=\{\lceil mr\rceil:r\in M\}.$$
More generally, $\lfloor mM\rfloor_\alpha:=\{\lfloor mr\rfloor_\alpha:r\in M\}.$

The first and second condition, respectively, in the definition of a tempered monoid imply that the set $\lfloor mM \rfloor_\alpha$ satisfies the first and second condition, respectively, in the definition of a numerical semigroup, for any $\alpha$.
However, the third condition in the definition of a tempered monoid $M$ does not guarantee in general the third condition for the set $\lfloor mM\rfloor_\alpha$ to be a numerical semigroup for a general positive integer $m$ and a real number in the unit interval $\alpha$.

\begin{example}
  Let $Q$ be the set in Example~\ref{e:quarterstempmonoid} and let $m=19$. Then,

\begin{eqnarray*}
  19 Q & = & \{ 0,\; 19,\;
 23.75,\; 28.5,\; 33.25,\; 38,\; 40.375,\; 42.75,\; 45.125,\; 47.5,\; 49.875,\; 52.25,\; \\&& 54.625,\; 57,\; 58.1875,\; 59.375,\; 60.5625,\; 61.75,\; 62.9375,\; 64.125,\; 65.3125,\; \\&& 66.5,\; 67.6875,\; 68.875,\; 70.0625,\; 71.25,\; 72.4375,\; 73.625,\; 74.8125,\; 76,\; \\&& 76.59375,\; 77.1875,\; 77.78125,\; 78.375,\; 78.96875,\; 79.5625,\; 80.15625,\;\; \dots\}\\
\big[ 19Q \big]&=&\{0,\; 19,\; 24,\; 29,\; 33,\; 38,\; 40,\; 43,\; 45,\; 48,\; 50,\; 52,\; 55,\; 57,\; 58,\; 59,\; 61,\; 62,\; 63,\; \\&&64,\; 65,\; 67,\; 68,\; 69,\; 70,\; 71,\; 72\}\cup\{i\in{\mathbb N}:i\geq 74\}\\
\lfloor19Q\rfloor&=&\{
0,\; 19,\; 23,\; 28,\; 33,\; 38,\; 40,\; 42,\; 45,\; 47,\; 49,\; 52,\; 54,\; 57,\; 58,\; 59,\; 60,\; 61,\; 62,\; \\&&64,\; 65,\; 66,\; 67,\; 68,\; 70,\; 71,\; 72,\; 73,\; 74\}\cup\{i\in{\mathbb N}:i\geq 76\}
\end{eqnarray*}

Notice that  $33,40\in[19Q]$, but $33+40=73\not\in[19Q]$, while  $28\in\lfloor 19Q\rfloor$, but $28+28=56\not\in\lfloor 19Q\rfloor$. Hence, neither $[19Q]$ nor $\lfloor 19Q\rfloor$ are numerical semigroups.
\end{example}

We say that a normalized tempered monoid $M$ is {\em discretizable by a multiplicity $m$ under rounding}, {\em under flooring}, {\em under ceiling} or, more generally, {\em under rounding by $\alpha$}, if $[mM]$, $\lfloor mM\rfloor$, $\lceil mM\rceil$,
or $\lfloor mM\rfloor_\alpha$, respectively, is a numerical semigroup. Notice that if the resulting integer set is a numerical semigroup, then $m$ is the multiplicity of the resulting numerical semigroup, in each case.
It is easy to check that every normalized tempered monoid is discretizable at least by $m=1$.
When $m=1$ the semigroup obtained is the semigroup of the naturals with the zero.

\begin{example}
  Let $Q$ be the set in Example~\ref{e:quarterstempmonoid} and let $m=16$. Then,
  \begin{eqnarray*}
    16Q&=&\{0, 16, 20, 24, 28, 32, 34, 36, 38, 40, 42, 44, 46, 48, 49, 50, 51, 52, 53, 54, 55,\\&&
56, 57, 58, 59, 60, 61, 62, 63, 64, 64.5, 65, 65.5, 66, 66.5, 67, 67.5, 68, 68.5, 69,\dots\}\\
  \big[16Q\big]&=&
\{0,16,20,24,28,32,34,36,38,40,42,44,46\}\cup\{i\in{\mathbb N}: i\geq 48\}\\
  \lfloor 16Q\rfloor&=&[16Q]
  \end{eqnarray*}

Since both $[16Q]$ and $\lfloor 16Q\rfloor$ are numerical semigroups,  $Q$ is discretizable by $16$ either by rounding or by flooring.
\end{example}

\section{Product-compatible tempered monoids}
In this section we introduce the property of {\em product-compatibility} and prove that there is only one normalized product-compatible tempered monoid. We call it the {\em logarithmic monoid}.
\begin{definition}
We say the tempered monoid $M=\{\mu_1, \mu_2, \mu_3, \dots\}$ with $\mu_i<\mu_{i+1}$ is {\em product-compatible} if $\mu_{ij}=\mu_{i}+\mu_{j}$ for any $i,j\in{\mathbb N}$.
For a justification of this definition, see the introduction and Figure~\ref{f:wpyth}.
\end{definition}

\begin{example}
  \label{e:logtempmonoid}
  The {\em logarithmic monoid} is the set $L=\{\log_2(i):i\in{\mathbb N}\}$.
  Denote $\lambda_1=\log_2(1)$, $\lambda_2=\log_2(2), \dots$
  See next the (rounded) smallest elements in $L$:

\begin{eqnarray*}
  L&=&\{0,\; 1,\; 1.5849,\; 2,\; 2.3219,\; 2.5849,\; 2.8073,\; 3,\; 3.1699,\; 3.3219,\; 3.4594,\; 3.5849,\; \\&& 3.7004,\; 3.8073,\; 3.9068,\; 4,\; 4.0874,\; 4.1699,\; 4.2479,\; 4.3219,\; 4.3923,\; 4.4594,\; \\&&4.5235,\; 4.5849,\; 4.6438,\; 4.7004,\; 4.7548,\; 4.8073,\; 4.8579,\; 4.9068,\; 4.9541,\; 5,\; \\&&5.0443,\; 5.0874,\; 5.1292,\; 5.1699,\; 5.2094,\; 5.2479,\; 5.2854,\; 5.3219,\; 5.3575,\;\dots\ \}
  \end{eqnarray*}

It is left to the reader to check that $L$ is a product-compatible tempered monoid.
It follows from the well-known property of logarithms stating that the logarithm of a product is the sum of logarithms. %That is, $\log_b(xy)=\log_b(x)+\log_b(y)$ for any positive $x,y,b\in{\mathbb R}$.
%Recall the well-known property of logarithms stating that the logarithm of a product is the sum of logarithms. That is, $\log_b(xy)=\log_b(x)+\log_b(y)$ for any positive $x,y,b\in{\mathbb R}$.
%First of all, one needs to check the three conditions of a tempered monoid.
%\begin{enumerate}
%\item The smallest element in $L$ is $\lambda_1=\log_2(1)=0$
%\item $\lambda_{n+1}-\lambda_n=\log_2(n+1)-\log_2(n)=\log_2(\frac{n+1}{n})$. Since the limit of $\frac{n+1}{n}$ is $1$, for any $\varepsilon>0$ one can take $n_0$ large enough so that $\frac{n+1}{n}<2^\varepsilon$ for any $n>n_0$. Then $\lambda_{n+1}-\lambda_n<\varepsilon$ for any $n>n_0$.
%\item For any $i,j\in{\mathbb N}$, one has $\lambda_i+\lambda_j=\log_2 i+\log_2 j=\log_2 ij=\lambda_{ij}\in L$.
%\end{enumerate}
%The tempered monoid $L$ is product-compatible because of the equalities
%$\lambda_{ij}=\log_2 (ij)=\log_2 (i) +\log_2 (j)=\lambda_{i}+\lambda_{j}$.
\end{example}

\begin{theorem}
The logarithmic monoid $L$ is the unique product-compatible normalized tempered monoid.
\end{theorem}

\begin{proof}
Suppose that $M=\{\mu_1, \mu_2, \mu_3, \dots\}$ with $\mu_i<\mu_{i+1}$
is a normalized product-compatible tempered monoid. One can check by
induction that
$\mu_{j^k}=k\mu_{j}$ for any $k\in{\mathbb N}$, $j\in{\mathbb N}$. 
%In particular, $\mu_2=1$ because $M$ is normalized. Since $M$ is product-compatible, $\mu_{ij}=\mu_{i}+\mu_{j}$ for any $i,j\in{\mathbb N}$. In particular, $\mu_{jj}=2\mu_{j}$ for any $j\in{\mathbb N}$ and, by induction, 
%\begin{eqnarray}\label{e:powlog}\mu_{j^k}&=&k\mu_{j}\end{eqnarray}
%for any $k\in{\mathbb N}$, $j\in{\mathbb N}$. 
Now, it is claimed that $\mu_i=\log_2(i)$  for any $i\in{\mathbb N}$.
Indeed, suppose that this does not hold for some $i\in{\mathbb N}$.
Then $\mu_i<\log_2(i)$ or $\mu_i>\log_2(i)$. Suppose first that
$\mu_i<\log_2(i)$. There exist $p,q\in{\mathbb N}$ such that
$\mu_i<p/q<\log_2(i)$. In particular, $q\mu_i<p$ while $q\log_2(i)>p$.
This, together with equality $\mu_{j^k}=k\mu_{j}$ 
implies that $\mu_{i^q}=q\mu_i<p=p\mu_2=\mu_{2^p}$, while $i^q=2^{q\log_2(i)}>2^p$, a contradiction since the sequence $\mu_1,\mu_2,\mu_3,\dots$ is supposed to be increasing.
An analogous contradiction is found for the case $\mu_i>\log_2(i)$.
\end{proof}

\section{Fractal monoids}

In this section we define {\em fractal monoids} and we prove that there is only one nonbisectional fractal monoid of granularity $2$. It is generated by the golden ratio and it is hence called the {\em golden fractal monoid}.

\begin{definition}
Let $M=\{\mu_1, \mu_2, \mu_3, \dots\}$ with $\mu_i<\mu_{i+1}$ be a normalized tempered monoid with granularity $\ell$. For any $i\in{\mathbb N}$, let $\ell_i$ be the cardinality of the $i$th period of $M$ and suppose that $$\pi_i(M)=\{i+\tau^{(i)}_0,i+\tau^{(i)}_1,\dots,i+\tau^{(i)}_{\ell_i-1}\},$$ with $\tau^{(i)}_0=0<\tau^{(i)}_1<\tau^{(i)}_2<\dots<\tau^{(i)}_{\ell_i-1}<\tau^{(i)}_{\ell_i}=1$.
We say that $M$ is {\em fractal} if for any $i\in{\mathbb N}$,
\begin{equation}
  \label{eq:genfractal}
  \pi_{i+1}(M)=\bigcup_{r=0}^{r<\ell_{i}}\bigcup_{s=0}^{s<\ell}\{(i+1)+\tau^{(i)}_r+\tau^{(1)}_s(\tau^{(i)}_{r+1}-\tau^{(i)}_r)\}.
\end{equation}
\end{definition}

Roughly speaking, we say that a tempered monoid is a {\em fractal monoid} if for any $i\in{\mathbb N}$
and for any interval between consecutive elements in the $i$th period,
the same interval appears in the next period (just adding $1$ to each end) divided exactly in the same portions as the interval from $1$ to $2$ is divided in the first period.

It is obvious that for each first period there exists exactly one such construction. So, for a fixed first period, there is at most one fractal monoid. Whether it exists or not will depend on whether the construction in \eqref{eq:genfractal} gives a set closed under addition or not. If this is the case, we say that the first period {\em generates} the tempered monoid as a fractal monoid.

%In the Appendix~\ref{a:algorithm} we included a {\sc Sage} code to construct fractal sequences and to check whether they are indeed monoids.

\mut{
\begin{example}\label{e:fractmonoidfirst}
  The period $\{1,1.5,1.75\}$ generates a fractal monoid.
  Its first period is, indeed $\{1,1.5,1.75\}$.
  In the second period, we have each of the elements in the first period plus one, i.e, $\{2,2.5,2.75\}$, and the three intervals in between, $[2,2.5]$, $[2.5,2.75]$, $[2.75,3]$ are subsequently divided using the same proportions as the first period. That is the interval
  $[2,2.5]$ is divided by $2.25$ and $2.375$, the interval
  $[2.5,2.75]$ is divided by $2.625$ and $2.6875$, and the interval
  $[2.75,3]$ is divided by     $2.875$ and $2.9375$.

  The smallest elements of the monoid are
  {\small$0.$, $1.$, $1.5$, $1.75$, $2.$, $2.25$, $2.375$, $2.5$, $2.625$, $2.6875$,
    $2.75$,
    $2.875$, $2.9375$,
    $3.$, $3.125$, $3.1875$, $3.25$, $3.3125$, $3.34375$, $3.375$, $3.4375$, $3.46875$, $3.5$, $3.5625$, $3.59375$, $3.625$, $3.65625$, $3.671875$, $3.6875$, $3.71875$, $3.734375$, $3.75$, $3.8125$, $3.84375$, $3.875$, $3.90625$, $3.921875$, $3.9375$, $3.96875$, $3.984375$, $4.$, $4.0625$, $4.09375$, $4.125$, $4.15625$, $4.171875$, $4.1875$, $4.21875$, $4.234375$, $4.25$, $4.28125$, $4.296875$, $4.3125$, $4.328125$, $4.3359375$, $4.34375$, $4.359375$, $4.3671875$, $4.375$, $4.40625$, $4.421875$, $4.4375$, $4.453125$, $4.4609375$, $4.46875$, $4.484375$, $4.4921875$, $4.5$, $4.53125$, $4.546875$, $4.5625$, $4.578125$, $4.5859375$, $4.59375$, $4.609375$, $4.6171875$, $4.625$, $4.640625$, $4.6484375$, $4.65625$, $4.6640625$, $4.66796875$, $4.671875$, $4.6796875$, $4.68359375$, $4.6875$, $4.703125$, $4.7109375$, $4.71875$, $4.7265625$, $4.73046875$, $4.734375$, $4.7421875$, $4.74609375$, $4.75$, $4.78125$, $4.796875$, $4.8125$, $4.828125$, $4.8359375$, $4.84375$, $4.859375$, $4.8671875$, $4.875$, $4.890625$, $4.8984375$, $4.90625$, $4.9140625$, $4.91796875$, $4.921875$, $4.9296875$, $4.93359375$, $4.9375$, $4.953125$, $4.9609375$, $4.96875$, $4.9765625$, $4.98046875$, $4.984375$, $4.9921875$, $4.99609375$, \dots}
\end{example}
}

\begin{example}
  The tempered monoid
$D=\{0\}\cup\{n+ \frac{k}{10^{n}} :n\in {\mathbb N} \mbox{ and } 0\leq k\leq 10^{n}-1\}$
  from Example~\ref{e:deustempmonoid} is a fractal monoid.
  It is generated by the period  $$\{1,1.1,1.2,1.3,1.4,1.5,1.6,1.7,1.8,1.9\}.$$
\end{example}

\begin{example}\label{e:nonfractmonoid}
Let us parallel the observation in Figure~\ref{f:l}, analyzing the period $\{1, 1+7/12\}$. If this period generated a fractal monoid, the monoid should be
\begin{eqnarray*}M&=& 
%\resizebox{.6\textwidth}{!}{
\{0,\ \ \ 1,\ \ \ 1+7/12,\ \ \ 
2,\ \ \ 2+(7/12)^2,\ \ \ 2+(7/12),\ \ \ 2+(7/12)+5\cdot 7/12^2,
%}
\\&&
%\resizebox{.6\textwidth}{!}{ 
3,\ \ \ 3+(7/12)^3,\ \ \ 3+(7/12)^2,\ \ \ 3+(7/12)^2+7^2\cdot 5/12^3,\
                                                                     \
                                                                     \
                                                                     3+(7/12),
%}
\\&&
%\resizebox{.6\textwidth}{!}{ 
3+(7/12)+5\cdot 7^2/12^3,\ \ \ 3+(7/12)+5\cdot 7/12^2,
\\&&
     3+(7/12)+5\cdot 7/12^2+5^2\cdot 7/12^3,
%}
\\&&\dots \}\\
&=& 
%\resizebox{.6\textwidth}{!}{
\{0,\ \ 1,\ \ 19/12,\ \ 2,\ \ 337/144,\ \ 31/12,\ \ 407/144,\ \ 3,\ \ 5527/1728,\ \
%}
\\&&
%\resizebox{.6\textwidth}{!}{ 
481/144,\ \ 6017/1728,\ \ 43/12,\ \ 6437/1728,\ \ 551/144,\ \ 6787/1728,\ \ 
\dots \}
%}
\\
\end{eqnarray*}
However, this is not a tempered monoid since it is not closed under addition. Indeed, $2(1+7/12)=3+2/12\not\in M$. Thus, $\{1, 1+7/12\}$ does not generate a fractal monoid.
\end{example}

\begin{example}
  The tempered monoid
  $L=\{\lambda_i=\log_2(i):i\in {\mathbb N}\}$
  in Example~\ref{e:logtempmonoid} is not fractal.
  Indeed, since its granularity is $2$, for it to be fractal it should satisfy
  $$\frac{\lambda_3-\lambda_2}{\lambda_4-\lambda_2}=\frac{\lambda_5-\lambda_4}{\lambda_6-\lambda_4}.$$ In that case one would have
  $$\frac{\log_2(3/2)}{\log_2(4/2)}=\frac{\log_2(5/4)}{\log_2(6/4)},$$ implying that
  $(\log_2(1.5))^2=\log_2(1.25)$, which is false.  
\end{example}

\begin{lemma}
The cardinality of the $i$th period of a fractal monoid of granularity $\ell$ is $\ell^i$, for any $i\in{\mathbb N}$.
\end{lemma}

\begin{example}
\label{e:perfectfractaltempmonoids}
For any integer $\ell\geq 2$, the period $\{1,1+1/\ell,\dots,1+(\ell-1)/\ell\}$ generates a fractal monoid whose $i$th period is $\pi_i=\{i,1+1/\ell^i,1+2/\ell^i,\dots,1+(\ell^i-1)/\ell^i\}$.
\end{example}

This example leads to the next definition.

\begin{definition}
  We say that the tempered monoid generated by
  the period $\{1,1+1/\ell,\dots,1+(\ell-1)/\ell\}$, for
  $\ell\geq 2$ is the {\em perfect fractal monoid of granularity $\ell$}. The perfect fractal monoid of granularity $2$ is called the {\em bisectional} fractal monoid.
\end{definition}

\begin{example}
The fractal monoid $D$ in Example~\ref{e:deustempmonoid} is a perfect fractal monoid.
\end{example}

We denote the golden ratio $\frac{1+\sqrt{5}}{2}$ by $\phi$. The proof
of the next theorem requires some lemmas and the whole
set has been moved to the end of the section.

\begin{theorem}\label{t:goldenfractal}
  \begin{enumerate}
    \item The period $\{1,\phi\}$ generates a fractal monoid. 
    \item The unique nonbisectional normalized fractal monoid of granularity $2$ is exactly the fractal monoid generated by the period $\{1,\phi\}$. 
  \end{enumerate}
\end{theorem}

\begin{definition}\label{d:goldenfractal}
The {\em golden fractal monoid} is the fractal monoid generated by the period $\{1,\phi\}$. It is denoted by $F$ and its terms are written $\varphi_1,\varphi_2,\varphi_3,\dots$ The first (rounded) terms are listed below. 

\begin{eqnarray*}
  F&=&\{0,
  1,
  1.6180, 2, 2.3820, 2.6180, 2.8541, 3, 3.2361, 3.3820, 3.5279, 3.6180,
  \\&&3.7639, 3.8541, 3.9443, 4, 4.1459, 4.2361, 4.3262, 4.3820, 4.4721, 4.5279, \\&&4.5836, 4.6180, 4.7082, 4.7639, 4.8197, 4.8541, 4.9098, 4.9443, 4.9787, 5, \\&&5.0902, 5.1459, 5.2016, 5.2361, 5.2918, 5.3262, 5.3607, 5.3820, 5.4377, \\&&5.4721, 5.5066, 5.5279, 5.5623, 5.5836, 5.6049, 5.6180, 5.6738, 5.7082, \\&&5.7426, 5.7639, 5.7984, 5.8197, 5.8409, 5.8541, 5.8885, 5.9098, 5.9311, \\&&5.9443, 5.9656, 5.9787, 5.9919, 6, \dots \}
\end{eqnarray*}
\end{definition}

The proof of Theorem~\ref{t:goldenfractal} will be preceded by three lemmas. For the lemmas we need some notation. 
Given $p\in(0,1)$, let $q=1-p$ and define 
$$\begin{array}{rrcl}
 f_0:& \{0\} &\rightarrow& [0,1)\\ 
  & 0 &\mapsto& 0\\\\
  f_\ell:& {\mathbb N}_0\cap [0,2^\ell) &\rightarrow& [0,1)\\ 
&n&\mapsto& \left\{\begin{array}{ll}p f_{\ell -1}(n) & \mbox{ if }n<2^{\ell-1},\\
p +q f_{\ell -1}(n-2^{\ell-1}) & \mbox{ if }n\geq 2^{\ell-1},\end{array}\right. 
\end{array}$$
for $\ell\in{\mathbb N}$. 
Notice that $f_\ell$ depends on the choice of $p$. 
The next lemma is a consequence of the definitions and describes the relationship between the maps $f_\ell$ and fractal monoids. 

\begin{lemma}\label{l:struct}
  If a tempered monoid $M=\{\mu_1,\mu_2,\mu_3,\dots\}$ with $\mu_i<\mu_{i+1}$ is fractal and has $\{1,1+p\}$ as its first period, then 
\begin{itemize}
  \item the elements of $M$ are exactly $\mu_i=\lfloor\log_2(i)\rfloor+f_{\lfloor\log_2(i)\rfloor}(i-2^{\lfloor\log_2(i)\rfloor})$,
\item 
  the $\ell$th period of $F$ is $\{\ell+f_\ell(0),\ell+f_\ell(1),\ell+f_\ell(2),\ell+f_\ell(3),\dots,\ell+f_\ell(2^\ell-1)\}$, with $\ell+f_\ell(0)<\ell+f_\ell(1)<\ell+f_\ell(2)<\ell+f_\ell(3)<\dots<\ell+f_\ell(2^\ell-1)$. 
  \end{itemize}
\end{lemma}

The next lemma states a simple result that will be used several times in 
the proofs of the following lemmas. It can be proved by induction. 

\begin{lemma}If $\ell,n\in{\mathbb N}_0$ with $\ell\geq 1$ and $n<2^{\ell-1}$, then $f_{\ell-1}(n)=f_{\ell}(2n)$. 
\end{lemma}

\mut{
\begin{proof}
  If $\ell=1$ the result is obvious. For $\ell>1$ we will proceed by induction. 
  \begin{eqnarray*}
    f_\ell(2n)&=&\left\{\begin{array}{ll}p f_{\ell -1}(2n) & \mbox{ if }2n<2^{\ell-1},\\p +q f_{\ell -1}(2n-2^{\ell-1}) & \mbox{ if }2n\geq 2^{\ell-1},\end{array}\right. 
    \\&=&\left\{\begin{array}{ll}p f_{\ell -1}(2n) & \mbox{ if }n<2^{\ell-2},\\p +q f_{\ell -1}(2n-2^{\ell-1}) & \mbox{ if }n\geq 2^{\ell-2},\end{array}\right. 
  \end{eqnarray*}
  Now, by the induction hypothesis, this equals 
  \begin{eqnarray*}
    \phantom{f_\ell(2n)}&=&\left\{\begin{array}{ll}p f_{\ell -2}(n) & \mbox{ if }n<2^{\ell-2},\\p +q f_{\ell -2}(n-2^{\ell-2}) & \mbox{ if }n\geq 2^{\ell-2},\end{array}\right.\\
&=&f_{\ell-1}(n). 
  \end{eqnarray*}
\end{proof}
}

%Recall that the golden ratio is $\phi=\frac{1+\sqrt{5}}{2}$. 
In the next lemma we take $p=\phi-1=\frac{-1+\sqrt{5}}{2}=0.618033\dots$. 
 We use that, in this case, $p^2=1-p$, $1=p+p^2$ and $2p=1+p^3$. 

\begin{lemma}\label{l:p}
  Let $p=\phi-1$, $q=1-p$, and let $\ell,n\in{\mathbb N}_0$ with 
  $\ell>0$ and $n<2^\ell$. The sum 
  $p+f_\ell(n)$ is one of 
\begin{itemize}
  \item $f_{\ell+1}(c)$ for some $0\leq c<2^{\ell+1}$,
  \item $1+pf_{\ell+1}(d)$ for some $0\leq d<2^{\ell+1}$. 
\end{itemize}
\end{lemma}

\begin{proof}

    If $\ell=1$ then $n$ is either $0$ or $1$ and $p+f_1(0)=p+0=p=f_1(1)$, while $p+f_1(1)=p+p=1+p^3=1+pf_2(1)$. 
Suppose that $\ell>1$ and assume that the lemma is true for $\ell-1$.  

If $n<2^{\ell-2}$, then $p+f_\ell(n)=p+p f_{\ell -1}(n)=p+p^2 f_{\ell-2}(n)=f_{\ell-1}(n+2^{\ell-2})=f_{\ell+1}(2^2n+2^{\ell})$, with $0\leq 2^2n+2^{\ell}<2^{\ell+1}$. 

If $2^{\ell-2}\leq n<2^{\ell-1}$, then $p+f_\ell(n)=p+pf_{\ell-1}(n)=p+p^2+p^3 f_{\ell-2}(n-2^{\ell-2})=1+p^3f_{\ell-2}(n-2^{\ell-2})=1+pf_{\ell}(n-2^{\ell-2})=1+pf_{\ell+1}(2n-2^{\ell-1})  $, with $0\leq 2n-2^{\ell-1}<2^{\ell+1}$. 

If $n\geq 2^{\ell-1}$, then $p+f_\ell(n)=2p +p^2 f_{\ell -1}(n-2^{\ell-1}) = 1+p^3 +p^2 f_{\ell -1}(n-2^{\ell-1}) = 1+p^2(p+f_{\ell-1}(n-2^{\ell-1}))$ with $n-2^{\ell-1}<2^{\ell-1}$. 
By the induction hypothesis, this either equals $1+p^2f_{\ell}(c')=1+pf_{\ell+1}(c')$ with $c'<2^\ell$, and hence, with $c'<2^{\ell+1}$, or $1+p^2(1+pf_{\ell}(d'))=1+p(p+p^2f_{\ell}(d'))=1+pf_{\ell+1}(d'+2^\ell)$ with $d'<2^\ell$, and thus, with $d'+2^\ell<2^{\ell+1}$. 

\end{proof}

\begin{lemma}\label{l:key}
Let $p=\phi-1$ and $q=1-p$. 
If $i,j,a,b\in{\mathbb N}_0$, with $0\leq a<2^i$ and $0\leq b<2^j$, then 
$f_i(a)+f_j(b)$
is one of 
\begin{itemize}
  \item $f_{i+j}(c)$ for some $0\leq c<2^{i+j}$,
  \item $1+f_{i+j+1}(d)$ for some $0\leq d<2^{i+j+1}$. 
    \end{itemize}
\end{lemma}

\begin{proof}
We proceed by induction on $i+j$. 
If $i+j$ equals $0$ then the result is obvious. 
Suppose that $i+j>0$. 
If one of $i$ and $j$ is $0$ then the result is also obvious. Therefore, we can assume that both $i$ and $j$ are non-zero. 
On one hand,
$f_i(a)=pf_{i-1}(a')$ or 
$f_i(a)=p+qf_{i-1}(a')$ for some $a'<2^{i-1}$. 
On the other hand,
$f_j(b)=pf_{j-1}(b')$ or 
$f_j(b)=p+qf_{j-1}(b')$ for some $b'<2^{j-1}$. 
Hence, one of the next cases holds. 
\begin{enumerate}
\item $f_i(a)+f_j(b)=pf_{i-1}(a')+pf_{j-1}(b')=p(f_{i-1}(a')+f_{j-1}(b'))$
  for some $a'<2^{i-1}$ and some $b'<2^{j-1}$. 
  By the induction hypothesis, this equals one of 
  \begin{itemize}
  \item $pf_{i+j-2}(c)$ for some $0\leq c<2^{i+j-2}$,
  \item $p(1+f_{i+j-1}(d))$ for some $0\leq d<2^{i+j-1}$. 
  \end{itemize}
  On one hand, $pf_{i+j-2}(c)=f_{i+j-1}(c)=f_{i+j}(2c)$, with $2c<2^{i+j}$. 
  On the other hand, $p(1+f_{i+j-1}(d))$ can be either 
  $p(1+pf_{i+j-2}(d'))$ or $p(1+p+qf_{i+j-2}(d'))$ for some $d'<2^{i+j-2}$. 
  But $p(1+pf_{i+j-2}(d'))=p+p^2f_{i+j-2}(d')=f_{i+j-1}(d'+2^{i+j-2})=f_{i+j}(2d'+2^{i+j-1})$ with $2d'+2^{i+j-1}<2^{i+j}$,
  while \break $p(1+p+qf_{i+j-2}(d'))=p+p^2+p^3f_{i+j-2}(d')=1+f_{i+j+1}(d')$, with $d'<2^{i+j+1}$. 

\item $f_i(a)+f_j(b)=pf_{i-1}(a')+p+qf_{j-1}(b')$
  for some $a'<2^{i-1}$ and some $b'<2^{j-1}$. 
  If $i=1$, then $a'=0$ and $f_i(a)+f_j(b)=p+qf_{j-1}(b')=f_j(b'+2^{j-1})=f_{i+j}(2^ib'+2^{i+j-1})$, with $2^ib'+2^{i+j-1}<2^{i+j}$. Therefore, we can assume $i>1$. 
  Now, by the definition of $f$, the sum $f_i(a)+f_j(b)$ equals one of 
  \begin{itemize}
  \item $p^2f_{i-2}(a')+p+qf_{j-1}(b')=p+p^2(f_{i-2}(a')+f_{j-1}(b'))$ if $a'<2^{i-2}$,
  \item 
 $p(p+qf_{i-2}(a'-2^{i-2}))+p+qf_{j-1}(b')=1+p^2(pf_{i-2}(a'-2^{i-2})+f_{j-1}(b'))$
    if $a'\geq 2^{i-2}$. 
  \end{itemize}
  The first sum, by the induction hypothesis, is either $p+q(f_{i+j-3}(c'))=f_{i+j-2}(c'+2^{i+j-3})=f_{i+j}(2^2c'+2^{i+j-1})$ with $c'<2^{i+j-3}$ and hence, with $2^2c'+2^{i+j-1}<2^{i+j}$, or $p+q(1+f_{i+j-2}(d'))=1+p^2f_{i+j-2}(d')=1+f_{i+j}(d')=1+f_{i+1+1}(2d')$, with $d'<2^{i+j-2}$ (and hence with $2d'<2^{i+j+1}$). 

    To analyze the second sum, notice that $pf_{i-2}(a'-2^{i-2})<p$ and, thus, $pf_{i-2}(a'-2^{i-2})+f_{j-1}(b')<1+p$. Now, taking this into account, and applying the induction hypothesis, we have 
that 
$1+p^2(pf_{i-2}(a'-2^{i-2})+f_{j-1}(b'))$ has the form 
$1+p^2(f_{i+j-2}(c'))=1+f_{i+j}(c')=1+f_{i+j+1}(2c')$ for some $c'<2^{i+j-2}$ (and hence $2c'<2^{i+j+1}$),
or 
$1+p^2(1+pf_{i+j-2}(c''))=1+p(p+qf_{i+j-2}(c''))=1+pf_{i+j-1}(c''+2^{i+j-2})=1+f_{i+j}(c''+2^{i+j-2})=1+f_{i+j+1}(2c''+2^{i+j-1})$ for some $c''<2^{i+j-2}$ (and hence $2c''+2^{i+j-1}<2^{i+j+1}$). 
  
\item $f_i(a)+f_j(b)=p+qf_{i-1}(a')+pf_{j-1}(b')$ for some $a'<2^{i-1}$ and some $b'<2^{j-1}$. This case can be proved as the previous one. 
\item $f_i(a)+f_j(b)=p+qf_{i-1}(a')+p+qf_{j-1}(b')=2p+q(f_{i-1}(a')+f_{j-1}(b'))$ for some $a'<2^{i-1}$ and some $b'<2^{j-1}$. 
By the induction hypothesis, this equals one of 
\begin{itemize}
\item $2p+qf_{i+j-2}(c')$ for some $0\leq c'<2^{i+j-2}$,
\item $2p+q(1+f_{i+j-1}(d'))$ for some $0\leq d'<2^{i+j-1}$. 
\end{itemize}
On one hand, $2p+qf_{i+j-2}(c')=p+f_{i+j-1}(c'+2^{i+j-2})$ and the result follows by Lemma~\ref{l:p}. 

On the other hand,
$2p+q(1+f_{i+j-1}(d'))=1+p+qf_{i+j-1}(d')=1+f_{i+j}(d'+2^{i+j-1})=1+f_{i+j+1}(2d'+2^{i+j})$ with $2d'+2^{i+j}<2^{i+j+1}$. 
\end{enumerate}
\end{proof}

Now we are ready to prove Theorem~\ref{t:goldenfractal}. 

\begin{proof}
  \begin{enumerate}
  \item One needs to see that, for the case when $p=\phi-1$, the set described in Lemma~\ref{l:struct} is closed under addition. This is a direct consequence of Lemma~\ref{l:key}. 
    \item 
      Suppose that $M$ is a normalized nonbisectional fractal monoid of granularity $2$. Then its first period is $\{1,1+p\}$ for some $p$ with $0<p<1$. 
      By Lemma~\ref{l:struct}, the second and third periods of the tempered monoid must be 
  $$\begin{array}{l}
    \{2,2+p^2,2+p,2+2p-p^2\},\\
      \{3,3+p^3,3+p^2,3+2p^2-p^3,3+p,3+p+p^2-p^3,\\\ \ \ \ \ 3+2p-p^2,3+3p-3p^2+p^3\},\\
  \end{array}$$
  where the elements in each period are presented in increasing order. 
  
  If $M$ is nonbisectional then $p\neq 0.5$. 
  For $M$ to be closed under addition, it must hold that $(1+p)+(1+p)=2+2p\in M$. 
  
  If $p<0.5$, then $2+p<2+2p<3$. Looking at the elements of the second period, one can deduce that $2+2p=2+2p-p^2$, leading to $p=0$, which is out of the range. 

  Then, $p$ must be larger than $0.5$. In this case, $3<2+2p<3+p$. Looking at the elements of the third period, one can deduce that $2+2p$ is either 
$3+p^3$, $3+p^2$, or $3+2p^2-p^3$. This leads to the equations
$p^3-2p+1=(p^2+p-1)(p-1)=0$, $p^2-2p+1=(p-1)^2=0$ or
$p^3-2p^2+2p-1=(p^2-p+1)(p-1)=0$. Among these equations, the unique
one having a real solution in the range $0.5 < p < 1$ is the first one, being the solution $p=\frac{-1+\sqrt{5}}{2}=\phi-1$. Hence, $1+p=\phi$. 
  \end{enumerate}
\end{proof}

\section{Odd-filterable tempered monoids and odd-filterable numerical semigroups}

As explained in the introduction and illustrated in Figure~\ref{f:waves} and in Figure~\ref{f:wodd} (b), cylindrical pipes with one open end and one closed end produce only the harmonics corresponding to odd multiples of the frequency of the fundamental. This suggests the definitions of {\em odd-filterable tempered monoids} and {\em odd-filterable numerical semigroups}. 

\begin{definition}
A tempered monoid is {\em odd-filterable} if the elements in its ordered sequence with odd index form a tempered monoid.
\end{definition}

\begin{theorem}
The logarithmic monoid $L$ is odd-filterable.
\end{theorem}

\begin{proof}
The logarithmic monoid $L$ is formed by the elements in the increasing sequence $(\lambda_i)_{i\in{\mathbb N}}=(\log_2(i))_{i\in {\mathbb N}}$.
We claim that the sequence of its terms with odd index
$(\bar\lambda_{i})_{i\in{\mathbb N}}=(\lambda_{2i-1})_{i\in{\mathbb N}}=(\log_2(2i-1))_{i\in {\mathbb N}}$ is a tempered monoid.
Let us check the three conditions of a tempered monoid.
\begin{enumerate}
\item Its smallest element is $\bar\lambda_1=\log_2(1)=0$
\item $\bar\lambda_{n+1}-\bar\lambda_n=\log_2(2n+1)-\log_2(2n-1)=\log_2(\frac{2n+1}{2n-1})$. Since the sequence $\frac{2n+1}{2n-1}$ converges to $1$, for any $\varepsilon>0$ one can take $n_0$ large enough so that $\frac{2n+1}{2n-1}<2^\varepsilon$ for any $n>n_0$. Then $\bar\lambda_{n+1}-\bar\lambda_n<\varepsilon$ for any $n>n_0$.
\item The third condition follows from the fact that the product of odd integers is again an odd integer.
%For any $i,j\in{\mathbb N}$, we have
%  \begin{eqnarray*}\bar\lambda_i+\bar\lambda_j&=&\log_2(2i-1)+\log_2(2j-1)
%    \\&=&\log_2((2i-1)(2j-1))=\log_2(4ij-2i-2j+1)
%    \\&=&\bar\lambda_{2ij-i-j+1},
%  \end{eqnarray*}
%  which is in the sequence.
\end{enumerate}
\end{proof}

However, the golden fractal monoid $F$ is not odd-filterable. Indeed, let $\tau=\phi-1$. Then
$\varphi_3=1+\tau$,
$\varphi_5=2+\tau^2$, while
the sum of them is $\varphi_3+\varphi_5=3+\tau+\tau^2=3+1=4=\varphi_{16}$, which has an even index in $F$.

The notion of odd-filterability for numerical semigroups is a bit more elaborate.
Let $S=\lfloor mM\rfloor_\alpha$, that is, $S$ is a numerical semigroup which arises from a tempered monoid $M=\{M_i\}_{i\in{\mathbb N}}$ which has been discretized by a multiplicity $m$ under rounding by $\alpha$. Then we define the {\em collapse} of $S$ with respect to $\{M,m,\alpha\}$ as the smallest integer $\kappa$ such that $\kappa=\lfloor m\mu_i\rfloor_\alpha=\lfloor m\mu_{i+1}\rfloor_\alpha$ for some $i$.

\begin{example}
\label{ex:H}
  Consider the unique nonbisectional fractal monoid $F$ of granularity $2$.
  Discretize it by multiplicity $12$ and by flooring. That is,
  \begin{eqnarray*}\lfloor 12F\rfloor&=&
    \lfloor 0.0000 \rfloor, \lfloor 12.0000 \rfloor, \lfloor 19.4164 \rfloor, \lfloor 24.0000 \rfloor, \lfloor 28.5836 \rfloor, \lfloor 31.4164 \rfloor, \lfloor 34.2492 \rfloor, \\&&\lfloor 36.0000 \rfloor, \lfloor 38.8328 \rfloor, \lfloor 40.5836 \rfloor, \lfloor 42.3344 \rfloor, \lfloor 43.4164 \rfloor, \lfloor 45.1672 \rfloor, \lfloor 46.2492 \rfloor,
    \\&&\lfloor 47.3313 \rfloor, \lfloor 48.0000 \rfloor, \lfloor 49.7508 \rfloor, \lfloor 50.8328 \rfloor, \lfloor 51.9149 \rfloor, \lfloor 52.5836 \rfloor, \lfloor 53.6656 \rfloor,
    \\&& \lfloor 54.3344 \rfloor, \lfloor 55.0031 \rfloor, \lfloor 55.4164 \rfloor, \lfloor 56.4984 \rfloor, \lfloor 57.1672 \rfloor, \lfloor 57.8359 \rfloor, \lfloor 58.2492 \rfloor,
    \\&& \lfloor 58.9180 \rfloor, \lfloor 59.3313 \rfloor, \lfloor 59.7446 \rfloor, \lfloor 60.0000 \rfloor, \lfloor 61.0820 \rfloor, \lfloor 61.7508 \rfloor,\dots
  \end{eqnarray*}
Notice that the resulting semigroup $\lfloor 12F\rfloor$ is the
well-tempered harmonic semigroup $H$ in \eqref{eq:H}.
In this case, the collapse is $55$, since 
  $\lfloor 12\mu_{23}\rfloor=\lfloor 55.0031 \rfloor= \lfloor 12\mu_{24}\rfloor=\lfloor 55.4164 \rfloor=55$ and $\lfloor 12\mu_i\rfloor\neq\lfloor 12 \mu_{i+1}\rfloor$ for any $i<23$.
\end{example}

\begin{definition}
  Suppose $S$ is a numerical semigroup of the form $S=\lfloor
  mM\rfloor_\alpha$. We say that $S$ is {\em odd-filterable with
    respect to $M,m,\alpha$} if the sum of any two elements in the
  semigroup with odd index is another element in the semigroup which 
either has odd index, or is bigger than or equal to the collapse.
%is either at least the collapse or which has odd index.
\end{definition}

\begin{example}\label{e:Hef}
Let us check that the well-tempered harmonic semigroup $H$ from \eqref{eq:H}, as a discretization of $F$ by flooring, is odd-filterable.
  Indeed, let us denote the elements of $H$ as follows:
  $h_1=0$, $h_2=12$, $h_3=19$, $h_4=24$, $h_5=28$, $h_6=31$, $h_7=34$, $h_8=36$, $h_9=38$, $h_{10}=40$, $h_{11}=42$, $h_{12}=43$,
  $h_{13}=45$, $h_{14}=46$, $h_{15}=47$, $h_{16}=48$, $h_{17}=49$, $h_{18}=50$, $h_{19}=51$, $h_{20}=52$, $h_{21}=53$, $h_{22}=54$, $h_{23}=55$, $h_{24}=56$, \dots

  Now, let us analyze all the sums of two elements with odd index.
  On one hand, $h_1+h_{2i-1}=h_{2i-1}$ for any $i\in{\mathbb N}$. If one of the summands is $h_3$, the options are 
  $h_3+h_3=h_9$, $h_3+h_5=h_{15}$, $h_3+h_7=h_{21}$, and $h_3+h_{2i-1}$, with $2i-1\geq 9$, which is larger than the collapse.
  For any other pair of summands, $h_{2i-1}+h_{2j-1}$, with $2i-1,2j-1\geq 5$, the sum is larger than the collapse.
\end{example}

\section{Emergence of the well tempered harmonic semigroup}

In this section we prove that the maximum number of equal divisions of the octave such that the discretizations of the golden fractal monoid $F$ and the logarithmic monoid $L$ coincide, and such that the discretization is odd-filterable is $12$. This is nothing else but the number of equal divisions of the octave in classical Western music.

\begin{theorem}
\label{t:pythfract}
  There exists a numerical semigroup of multiplicity $m$ that is simultaneously a discretization of the logarithmic monoid $L$ and a discretization of the golden fractal monoid $F$ if and only if $m\in\{1,2,3,4,5,6,7,8,9,10,12,13,18\}$.
  %  For an integer $m$, there exist $\alpha_L^{(m)}, \alpha_F^{(m)}$ such that $\lfloor mL\rfloor_{\alpha_L^{(m)}}=\lfloor mF\rfloor_{\alpha_F^{(m)}}$ and such that the obtained set is a numerical semigroup if and only if
\end{theorem}

\begin{proof}
  In Table~\ref{t:dis} 
we give a numerical semigroup of each multiplicity, whenever it exists, such that it is a simultaneous discretization of $L$ and $F$.

Let us prove now that there are no $\alpha_L$, $\alpha_F$ such that $\lfloor 11 L\rfloor_{\alpha_L}=\lfloor 11 F\rfloor_{\alpha_F}$, and $\lfloor 11 L\rfloor_{\alpha_L}$ is a numerical semigroup.
Indeed, suppose that, conversely, $\alpha_L$, $\alpha_F$ satisfy both conditions.
Observe that,
  $$\begin{array}{ll}  
  11\lambda_2=11 & 11\varphi_2=11\\
  11\lambda_3=17.4346 & 11\varphi_3=17.7984\\
  11\lambda_4=22 & 11\varphi_4=22
\end{array}$$
where the real values have been rounded.
The fact that $\lfloor 11 L\rfloor_{\alpha_L}=\lfloor 11 F\rfloor_{\alpha_F}$ implies that $\lfloor 11 \lambda_3\rfloor_{\alpha_L}=\lfloor 17.4346\rfloor_{\alpha_L}$ equals $\lfloor 11 \varphi_3\rfloor_{\alpha_F}=\lfloor 17.7984\rfloor_{\alpha_F}$, which in turn must be either $17$ or $18$. 
Since $11\varphi_8=33$ and $11\varphi_9=35.5967$, we deduce that $17+17=34\not\in \lfloor 11F\rfloor_{\alpha_F}$, and so $17$ can not be in the discretization.
On the other hand, since
$$\begin{array}{ll}  
11\lambda_8=33&	11\varphi_8=33\\
11\lambda_9=34.8692&	11\varphi_9=35.5967\\
11\lambda_{10}=36.5412&	11\varphi_{10}=37.2016\\
\end{array}$$
we deduce that $18+18=36$ does not belong to the discretization. So, the discretization is not a numerical semigroup.

\mut{
Let us prove now that there are no $\alpha_L$, $\alpha_F$ such that $\lfloor 14 L\rfloor_{\alpha_L}=\lfloor 14 F\rfloor_{\alpha_F}$, and $\lfloor 14 L\rfloor_{\alpha_L}$ is a numerical semigroup.
Indeed, suppose that, conversely, $\alpha_L$, $\alpha_F$ satisfy both conditions.
Observe that,
  $$\begin{array}{ll}  
  14\lambda_4=28 & 14\varphi_4=28\\
  14\lambda_5=  32.5070  & 14\varphi_5=33.3475\\
\end{array}$$
The fact that $\lfloor 14 L\rfloor_{\alpha_L}=\lfloor 14 F\rfloor_{\alpha_F}$ implies that
$\lfloor 14 \lambda_5\rfloor_{\alpha_L}=\lfloor 32.5070\rfloor_{\alpha_L}$ equals $\lfloor 14 \varphi_5\rfloor_{\alpha_F}=\lfloor 33.3475\rfloor_{\alpha_F}$, which in turn must equal $33$.
This means that
$\alpha_F> 0.3475$.
Now, let us consider 
$\lfloor 14\varphi_{3}\rfloor_{\alpha_F}=\lfloor 22.6525\rfloor_{\alpha_F}$. It may equal either
$22$ or $23$.
But, since $\alpha_F> 0.3475$,
\begin{eqnarray*}
  \lfloor 14\varphi_{8}\rfloor_{\alpha_F}=\lfloor  42\rfloor_{\alpha_F}&=&42,\\
  \lfloor 14\varphi_{9}\rfloor_{\alpha_F}=\lfloor 45.3050\rfloor_{\alpha_F}&=&45,\\
  \lfloor 14\varphi_{10}\rfloor_{\alpha_F}=\lfloor 47.3475\rfloor_{\alpha_F}&=&47,\\
\end{eqnarray*}
and so,
it happens that $2\lfloor 14\varphi_{3}\rfloor_{\alpha_F}$ does not belong to $\lfloor 14 F\rfloor_{\alpha_F}$.
So, $\lfloor 14 F\rfloor_{\alpha_F}$ is not a numerical semigroup.

Let us prove now that there are no $\alpha_L$, $\alpha_F$ such that $\lfloor 15 L\rfloor_{\alpha_L}=\lfloor 15 F\rfloor_{\alpha_F}$, and $\lfloor 15 L\rfloor_{\alpha_L}$ is a numerical semigroup.
Indeed, suppose that, conversely, $\alpha_L$, $\alpha_F$ satisfy both conditions.
Observe that,
  $$\begin{array}{ll}  
  15\lambda_4=30 & 15\varphi_4=30\\
  15\lambda_5=34.8289& 15\varphi_5=35.7295\\
\end{array}$$
The fact that $\lfloor 15 L\rfloor_{\alpha_L}=\lfloor 15 F\rfloor_{\alpha_F}$ implies that
$\lfloor 15 \lambda_5\rfloor_{\alpha_L}=\lfloor 34.8289\rfloor_{\alpha_L}$ equals $\lfloor 15 \varphi_5\rfloor_{\alpha_F}=\lfloor35.7295\rfloor_{\alpha_F}$, which in turn must equal $35$.
This means that $\alpha_L\leq 0.8289$ and that
$\alpha_F> 0.7295$.
In particular,
$$\lfloor 15\lambda_{21}\rfloor_{\alpha_L}=\lfloor 65.8848\rfloor_{\alpha_L}=66.$$
But,
\begin{eqnarray*}
  \lfloor 15\varphi_{20}\rfloor_{\alpha_F}=\lfloor
  65.7295\rfloor_{\alpha_F}&=&65,\\
  \lfloor 15\varphi_{21}\rfloor_{\alpha_F}=\lfloor
67.0820\rfloor_{\alpha_F}&=&67,
\end{eqnarray*}
and so $66$ belongs to $\lfloor 15L\rfloor_{\alpha_L}$ but
it does not belong to $\lfloor 15F\rfloor_{\alpha_F}$, contradicting the hypotheses.
}

Similar arguments show that for $m=14,15,16,17$, and for $m>18$,
there are no $\alpha_L$, $\alpha_F$ such that $\lfloor m L\rfloor_{\alpha_L}=\lfloor m F\rfloor_{\alpha_F}$, and $\lfloor m L\rfloor_{\alpha_L}$ is a numerical semigroup.
As an example, for $m=34$,
$34\lambda_5=78.9456$ while $34\varphi_5=80.9868$.
For multiplicities larger than $34$, it holds that
$m\varphi_4=m\lambda_4=2m$, while 
$m\varphi_5>m\lambda_5+2>2m$.
So, to have $\lfloor m L\rfloor_{\alpha_L}=\lfloor m F\rfloor_{\alpha_F}$,
one needs $\lfloor m \lambda_5\rfloor_{\alpha_L}=\lfloor m \varphi_5\rfloor_{\alpha_F}$, which is impossible since $m\varphi_5>m\lambda_5+2$.
\end{proof}

\begin{table}
{\footnotesize
\begin{itemize}
\item $\lfloor 1 L\rfloor_{0.50}=\lfloor 1 F\rfloor_{1.00}=\{0, 1, \dots\}$
\item $\lfloor 2 L\rfloor_{0.50}=\lfloor 2 F\rfloor_{1.00}=\{0, 2, 3, \dots\}$
\item $\lfloor 3 L\rfloor_{0.50}=\lfloor 3 F\rfloor_{0.85}=\{0, 3, 5, 6, 7, 8, \dots\}$
\item $\lfloor 4 L\rfloor_{0.28}=\lfloor 4 F\rfloor_{0.47}=\{0, 4, 7, 8, 10, 11, 12, 13, 14, \dots\}$
\item $\lfloor 5 L\rfloor_{0.50}=\lfloor 5 F\rfloor_{0.90}=\{0, 5, 8, 10, 12, 13, 14, 15, 16, 17, \dots\}$
\item $\lfloor 6 L\rfloor_{0.01}=\lfloor 6 F\rfloor_{0.41}=\{0, 6, 10, 12, 14, 16, 17, 18, 20, 21, 22, 23, 24, 25, 26, \dots\}$
\item $\lfloor 7 L\rfloor_{0.50}=\lfloor 7 F\rfloor_{0.97}=\{0, 7, 11, 14, 16, 18, 20, 21, 22, 23, 24, 25, 26, 27, \dots\}$
\item $\lfloor 8 L\rfloor_{0.35}=\lfloor 8 F\rfloor_{0.83}=\{0, 8, 13, 16, 19, 21, 23, 24, 26, 27, 28, 29, 30, 31, 32, 33, 34, \dots\}$
\item $\lfloor 9 L\rfloor_{0.13}=\lfloor 9 F\rfloor_{0.56}=$\\\ \ $=  \{0, 9, 15, 18, 21, 24, 26, 27, 29, 30, 32, 33, 34, 35, 36, 37, 38, 39, 40, 41, \dots\}$
\item $\lfloor 10 L\rfloor_{0.50}=\lfloor 10 F\rfloor_{1.00}=$\\\ \ $=\{0, 10, 16, 20, 23, 26, 28, 30, 32, 33, 35, 36, 37, 38, 39, 40, 41, 42, 43, 44, 45, \dots\}$
\item $\lfloor 12 L\rfloor_{0.40}=\lfloor 12 F\rfloor_{1.00}=$\\\ \ $=\{0, 12, 19, 24, 28, 31, 34, 36, 38, 40, 42, 43, 45, 46, 47, 48, 49, 50, 51, 52, 53, 54, 55, 56,$ $57, \dots\}$
\item $\lfloor 13 L\rfloor_{0.18}=\lfloor 13 F\rfloor_{0.94}=$\\\ \ $=\{0, 13, 21, 26, 31, 34, 37, 39, 42, 44, 45, 47, 48, 50, 51, 52, 53, 55, 56, 57, 58, 59, 60, 61, 62, 63,$ $64, 65, 66, 67, 68, \dots\}$
\item $\lfloor 18 L\rfloor_{0.05}=\lfloor 18 F\rfloor_{0.88}=$\\\ \ $=\{0, 18, 29, 36, 42, 47, 51, 54, 58, 60, 63, 65, 67, 69, 71, 72, 74, 76, 77, 78, 80, 81, 82, 83, 84, 85,$ $86, 87, 88, 89, 90, 91, 92, 93, 94, 95, 96, 97, 98, \dots\}$
\end{itemize}}
\caption{Numerical semigroups of each multiplicity, whenever they exist, such that they are simultaneous discretizations of $L$ and $F$.}
\label{t:dis}
\end{table}

\mut{
\begin{lemma}
  Notacio antiga
  There are no $\alpha_L$, $\alpha_F$ such that
  \begin{itemize}
  \item $\lfloor 16 L\rfloor_{\alpha_L}=\lfloor 16 F\rfloor_{\alpha_F}$
  \item $\lfloor 16 L\rfloor_{\alpha_L}$ is a numerical semigroup.
  \end{itemize}
\end{lemma}

\begin{proof}
Suppose that $\alpha_L$, $\alpha_F$ are as in the statement.
On one hand,
$$\begin{array}{ll}
16\lambda_3=32 & 16\varphi_3=32\\
16\lambda_4=37.150850 & 16\varphi_4=38.111456\\
16\lambda_5=41.359400 & 16\varphi_5=41.888544
\end{array}.$$
The fact that $\lfloor 16 L\rfloor_{\alpha_L}=\lfloor 16 F\rfloor_{\alpha_F}$ implies that
$\lfloor 16 \lambda_4\rfloor_{\alpha_L}=\lfloor 37.150850\dots\rfloor_{\alpha_L}$ equals $\lfloor 16 \varphi_4\rfloor_{\alpha_F}=\lfloor 38.111456\dots\rfloor_{\alpha_F}$, which in turn must equal $38$.
In particular, this means that $\alpha_L\leq 0.150850\dots$. In particular,
$$\lfloor 16\lambda_2\rfloor_{\alpha_L}=\lfloor 25.359400\rfloor_{\alpha_L}=26.$$
But
\begin{eqnarray*}
  \lfloor 16\lambda_8\rfloor_{\alpha_L}=\lfloor 50.718800\dots\rfloor_{\alpha_L}&=&51\\
  \lfloor 16\lambda_9\rfloor_{\alpha_L}=\lfloor 53.150850\dots\rfloor_{\alpha_L}&=&54
\end{eqnarray*}
Hence, $\lfloor 16\lambda_2\rfloor_{\alpha_L}+ \lfloor 16\lambda_2\rfloor_{\alpha_L}$ does not belong to $\lfloor 16 L\rfloor_{\alpha_L}$, meaning that $\lfloor 16 L\rfloor_{\alpha_L}$ is not a numerical semigroup, and so contradicting the hypothesis.
  \end{proof}

\begin{lemma}
  Notacio antiga
  There are no $\alpha_L$, $\alpha_F$ such that
  \begin{itemize}
  \item $\lfloor 17 L\rfloor_{\alpha_L}=\lfloor 17 F\rfloor_{\alpha_F}$
  \item $\lfloor 17 L\rfloor_{\alpha_L}$ is a numerical semigroup.
  \end{itemize}
\end{lemma}

\begin{proof}
2: 26.944363->[26-27] 	27.506578->[27-28] 	27 
7: 51.000000->51 	51.000000->51 		51 (enter)
8: 53.888725->[53-54] 	55.013156->[55-56] 	55 
\end{proof}

\begin{lemma}
  Notacio antiga
  There are no $\alpha_L$, $\alpha_F$ such that
  \begin{itemize}
  \item $\lfloor 19 L\rfloor_{\alpha_L}=\lfloor 19 F\rfloor_{\alpha_F}$
  \item $\lfloor 19 L\rfloor_{\alpha_L}$ is a numerical semigroup.
  \end{itemize}
\end{lemma}

\begin{proof}
alpha=0.116634 		 alpha=0.772063
19*PL[15]=76.000000 	 19*FM[15]=76.000000 
19*PL[16]=77.661794 	 19*FM[16]=78.772063 
19*PL[17]=79.228575 	 19*FM[17]=80.485292 

1: 19.000000->19 	19.000000->19 		19 (enter)
2: 30.114288->[30-31] 	30.742646->30 30 V(0.742646 0.772063) 
3: 38.000000->38 	38.000000->38 		38 (enter)

30+30=60

7: 57.000000->57 	57.000000->57 		57 (enter)
8: 60.228575->61 	61.485292->61 	61 
\end{proof}

\begin{lemma}
  Notacio antiga
  There are no $\alpha_L$, $\alpha_F$ such that
  \begin{itemize}
  \item $\lfloor 20 L\rfloor_{\alpha_L}=\lfloor 20 F\rfloor_{\alpha_F}$
  \item $\lfloor 20 L\rfloor_{\alpha_L}$ is a numerical semigroup.
  \end{itemize}
\end{lemma}

\begin{proof}
11: 71.699250->72 	72.360680->[72-73] 	72 
12: 74.008794->[74-75] 	75.278640->75 	75 
13: 76.147098->[76-77] 	77.082039->77 	77 
implies $\alpha_L\leq 0.008794$.

15: 80.000000->80 	80.000000->80 		80 (enter)
16: 81.749257->82 	82.917961->[82-83] 	82 
17: 83.398500->[83-84] 	84.721360->[84-85] 	84 
implies $\alpha_F\geq 0.917961$.

13: 76.147098->[76-77] 	77.082039->77 	77 
14: 78.137812->[78-79] 	78.885438->[78-79] [78 79] 
15: 80.000000->80 	80.000000->80 		80 (enter)
implies $\lfloor20\lambda_{14}\rfloor_{\alpha_L}=
\lfloor20\varphi_{14}\rfloor_{\alpha_F}$.
But, because of the bounds on $\alpha_L$ and $\alpha_F$,
it holds $\lfloor20\lambda_{14}\rfloor_{\alpha_L}=79$, while 
$\lfloor20\varphi_{14}\rfloor_{\alpha_F}=78$. 
\end{proof}

m=20
15: 80.000000->80 	80.000000->80 		80 (enter)
16: 81.749257->82 	82.917961->82 	82 
17: 83.398500->84 	84.721360->84 	84 
-> .91 ha de baixar

11: 71.699250->72 	72.360680->72 	72 
12: 74.008794->75 	75.278640->75 	75 
13: 76.147098->77 	77.082039->77 	77 
-> .008 ha de pujar

13: 76.147098->77 	77.082039->77 	77 
14: 78.137812->79 	78.885438->78 79 A(0.137812 0.008793) 
15: 80.000000->80 	80.000000->80 		80 (enter)

m=21,
2: 33.284213->34 	33.978714->[33-34] 34 A(0.284213 0.017600) 

7: 63.000000->63 	63.000000->63 		63 (enter)
8: 66.568425->67 	67.957428->[67-68] 	67 
9: 69.760490->70 	71.021286->71 	71 

m=22,
7: 66.000000->66 	66.000000->66 		66 (enter)
8: 69.738350->[69-70] 	71.193496->[71-72] 	71 
9: 73.082418->[73-74] 	74.403252->[74-75] 	74 

m=23
7: 69.000000->69 	69.000000->69 		69 (enter)
8: 72.908275->[72-73] 	74.429563->[74-75] 	74 
9: 76.404346->[76-77] 	77.785218->[77-78] 	77 

m=24
3: 48.000000->48 	48.000000->48 		48 (enter)
4: 55.726274->[55-56] 	57.167184->[57-58] 	57 
5: 62.039100->[62-63] 	62.832816->[62-63] [62 63] 

m=25

9: 83.048202->[83-84] 	84.549150->[84-85] 	84 
10: 86.485790->[86-87] 	88.196601->[88-89] 	88 
11: 89.624063->[89-90] 	90.450850->[90-91] 	90

m=26
5: 67.209025->[67-68] 	68.068884->[68-69] 	68 
6: 72.991228->[72-73] 	74.206651->[74-75] 	74 
7: 78.000000->78 	78.000000->78 		78 (enter)

m=27
3: 54.000000->54 	54.000000->54 		54 (enter)
4: 62.692059->[62-63] 	64.313082->[64-65] 	64 
5: 69.793988->[69-70] 	70.686918->[70-71] 	70 

m=28
7: 84.000000->84 	84.000000->84 		84 (enter)
8: 88.757900->[88-89] 	90.609903->[90-91] 	90 
9: 93.013987->[93-94] 	94.695048->[94-95] 	94 

m=29
3: 58.000000->58 	58.000000->58 		58 (enter)
4: 67.335915->[67-68] 	69.077014->[69-70] 	69 
5: 74.963913->[74-75] 	75.922986->[75-76] 	75 

m=30
3: 60.000000->60 	60.000000->60 		60 (enter)
4: 69.657843->[69-70] 	71.458980->[71-72] 	71 
5: 77.548875->[77-78] 	78.541020->[78-79] 	78 

m=31
3: 62.000000->62 	62.000000->62 		62 (enter)
4: 71.979771->[71-72] 	73.840946->[73-74] 	73 
5: 80.133838->[80-81] 	81.159054->[81-82] 	81 

m=32
3: 64.000000->64 	64.000000->64 		64 (enter)
4: 74.301699->[74-75] 	76.222912->[76-77] 	76 
5: 82.718800->[82-83] 	83.777088->[83-84] 	83 

m=33
4: 76.623627->[76-77] 	78.604878->[78-79] 	78 

m=34
4: 78.945555->[78-79] 	80.986844->[80-81] 	80 
}

\begin{theorem}
\label{t:pythfractef}
  There exists an odd-filterable numerical semigroup of multiplicity $m$ that is simultaneously a discretization of the logarithmic monoid $L$ and a discretization of the golden fractal monoid $F$ if and only if
$m\in\{1,2,3,4,5,6,7,8,10,12\}$.
\end{theorem}

\begin{proof}
  First of all, we leave to the reader to check that the numerical semigroups
in Table~\ref{t:dis},
except for
$\lfloor 9 L\rfloor_{0.13}=\lfloor 9 F\rfloor_{0.56}$,
$\lfloor 13 L\rfloor_{0.18}=\lfloor 13 F\rfloor_{0.94}$, and
$\lfloor 18 L\rfloor_{0.05}=\lfloor 18 F\rfloor_{0.88}$, are odd-filterable.
The case of flooring of $F$ with multiplicity $12$ is shown in Example~\ref{e:Hef}, and the other cases can be shown in a similar way.
  
Let us show now that there is no odd-filterable numerical semigroup of multiplicity $9$ that is simultaneously a discretization of $L$ and a discretization of $F$. Indeed, suppose that $S=\lfloor 9L\rfloor_{\alpha_L}=\lfloor 9F\rfloor_{\alpha_F}$ is a numerical semigroup for some $\alpha_L$ and some $\alpha_F$.
Let $S=\{s_1=0, s_2,s_3,\dots\},$ with $s_i<s_{i+1}$ for all $i \in {\mathbb N}$.
%By the table in Appendix~\ref{a:tables},
We deduce that $s_3$ must be either $14$ or $15$. Since $28$ is not in $\lfloor 9F\rfloor_\alpha$ for any $\alpha$, $s_3$ must be $15$. Then $30$ must be in $\lfloor 9F\rfloor_{\alpha_F}$  and the unique option is $30=\lfloor 9\lambda_{10}\rfloor_{\alpha_L}=\lfloor 9\varphi_{10}\rfloor_{\alpha_F}$, which has even index and is smaller than the collapse. Then the semigroup can not be odd-filterable.

Similarly, it can be shown that there is no odd-filterable numerical
semigroup of multiplicity $13$ or $18$ that is simultaneously a
discretization of $L$ and a discretization of $F$, and it is left to
the reader.

\mut{Let us show now that there is no odd-filterable numerical semigroup of multiplicity $13$ that is simultaneously a discretization of $L$ and a discretization of $F$. Indeed,
suppose that $S=\lfloor 13L\rfloor_{\alpha_L}=\lfloor 13F\rfloor_{\alpha_F}$ is a numerical semigroup for some $\alpha_L$ and some $\alpha_F$.
Let $S=\{s_1=0, s_2,s_3,\dots\},$ with $s_i<s_{i+1}$ for all $i \in {\mathbb N}$.
Observe that 
$$\begin{array}{ll}
13\lambda_{8}=39	 & 13\varphi_{9}=39\\
13\lambda_{9}=41.2090	 & 13\varphi_{9}=42.0689\\
\end{array}$$
implies 
$s_8=42$ and so $\alpha_L \leq 0.2090$. 
This, together with 
$$
\begin{array}{ll}
13\lambda_{10}=43.1851	 & 13\varphi_{10}=43.9656\\ 
13\lambda_{11}=44.9726	 & 13\varphi_{11}=45.8622\\
\end{array}
$$
implies $s_{11} =45$ and so $\alpha_F> 0.8622$.
From these bounds on $\alpha_L$ and $\alpha_F$ 
%and the values
%in Appendix~\ref{a:tables},
we deduce that either 
\begin{eqnarray}   
S&=&\{0, 13, 21, 26, s_4, 34, 37, 39, 42, s_9, 45, 47,s_{12}, 50,  51,
     52, s_{16}, 55, 56, \label{nocincset}\\&&58, 59, 60,  61, 62, 63, 64, 65, 66, 67, 68,
     69, 70, 71, 72, 73, 74, 75, 76,  77,\dots \},\nonumber
\end{eqnarray} 
or 
\begin{eqnarray*}   
S&=&\{0, 13, 21, 26, s_4, 34, 37, 39, 42, s_9, 45, 47,s_{12}, 50,  51,
     52, s_{16}, 55, 56, 57, \\&&58, 59, 60,  61, 62, 63, 64, 65, 66, 67,
     68, 69, 70, 71, 72, 73, 74, 75, 76,  77,\dots \},
\end{eqnarray*} 
where $s_5$ is $30$ or $31$,
$s_{10}$ is  $43$ or $44$,
$s_{13} $ is $48$ or $49$,
$s_{17} $ is $53$ or $54$. 

Suppose first that $s_5=30$. 
Since $13\lambda_{5}=30.1851$, $13\varphi_{5}=30.9656$, this implies 
$\alpha_L>0.1851$ and $\alpha_F >0.9656$. 
Now, $\alpha_L>0.1851$ implies that $\lfloor  13 \lambda_{21}\rfloor_{\alpha_L } =\lfloor 57.1001 \rfloor_{\alpha_L } =57\in S$. 
But $\alpha_F >0.9656$, together with
$$\begin{array}{l}
13\varphi_{20}=56.9656\\13\varphi_{21}=58.1378\end{array}$$
implies that 
$57\not\in \lfloor  13 F\rfloor_{\alpha_F}$, a contradiction. Hence,
$s_5=31$. 

Suppose, consequently, that $s_5=31$. 
In particular, $57=31+26$ needs to be in $S$ and so we discard the 
case \eqref{nocincset}. 
By semigroup properties, $s_{10}=44$, since it equals $13+31$. 

Suppose that $s_{13} =48$. Since  $13\varphi_{13}=48.9311$, this 
implies  $\alpha_F >0.9311$, which, in turn, 
implies $s_{17} =\lfloor 13\varphi_{17}\rfloor_{\alpha_F}=\lfloor  53.8967 \rfloor_{\alpha_F}=53$. 
So, we get, 
$$S=\{0, 13, 21, 26, 31, 34, 37, 39, 42, 44, 45, 47,48, 50,  51, 52,
53, 55, 56, 57, 58, 59, 60,\dots\}.$$
It is easy to check that this semigroup is
$\lfloor 13L\rfloor_{.18}=\lfloor 13F\rfloor_{.94}$.

Suppose now that $s_{13} =49$. Since 
$13\lambda_{13}=48.1057$, this 
implies  $\alpha_L \leq 0.1057$, which, in turn, 
implies $s_{17} =\lfloor 13\lambda_{17}\rfloor_{\alpha_L}=\lfloor  53.1370\rfloor_{\alpha_L}=54$. 
So, we get, 
$$S=\{0, 13, 21, 26, 31, 34, 37, 39, 42, 44, 45, 47,49, 50,  51, 52,54, 55, 56, 57, 58, 59, 60,\dots\}.$$
In this case, the semigroup is 
$\lfloor 13L\rfloor_{.1}=\lfloor 13F\rfloor_{.89}$. 

None of the two previous semigroups are odd-filterable numerical 
semigroups since in both cases 
$s_3+s_5=21+31=52=s_{16}$. 
So, there is no odd-filterable numerical semigroup of multiplicity $13$ that is simultaneously a discretization of $L$ and a discretization of $F$.

Let us finally show that there is no odd-filterable numerical semigroup of multiplicity $18$ that is simultaneously a discretization of $L$ and a discretization of $F$. Indeed, suppose that $S=\lfloor 18L\rfloor_{\alpha_L}=\lfloor 18F\rfloor_{\alpha_F}$ is a numerical semigroup for some $\alpha_L$ and some $\alpha_F$.
Let $S=\{s_1=0, s_2,s_3,\dots\},$ with $s_i<s_{i+1}$ for all $i \in{\mathbb N}$.
%Observe the table in Appendix~\ref{a:tables}.
%On one hand, 
Observe that on one hand
$$\begin{array}{ll} 18\lambda_{16}=72	 & 18\varphi_{16}=72\\ 18\lambda_{17}=73.5743	 & 18\varphi_{17}=74.6262\\ 18\lambda_{18}=75.0587	 & 18\varphi_{18}=76.2492 \end{array}$$ implies $\alpha_L \leq 0.0587$. 
This already implies that the unique option for $S$ is
the semigroup $\lfloor 9 L\rfloor_{0.13}=\lfloor 9 F\rfloor_{0.56}$ given in the proof of Theorem~\ref{t:pythfract}.
Moreover, $\alpha_L \leq 0.0587$ % together with the values in
                                % Appendix~\ref{a:tables} imply
implies 
that the collapse of $S$  with respect to $L$ is $90$. On the other hand, $$\begin{array}{ll} 18\lambda_{4}=36	 & 18\varphi_{4}=36\\ 18\lambda_{5}=41.7947	 & 18\varphi_{5}=42.8754\end{array}$$ implies $\alpha_F > 0.8754$ and this in turn implies that the collapse of $S$  with respect to $F$ is at least $88$. Now, $S$ is not an odd-filterable numerical semigroup with respect to $L$, $F$, since $s_3+s_9=29+58=87=s_{28}$ and $s_{28}$ is smaller than the collapse in both cases.}
\end{proof}    

Let us finally check that the numerical semigroup $H$ in the introduction, which equals $\lfloor 12 L\rfloor_{0.40}=\lfloor 12 F\rfloor_{1.00}$ in the proof of Theorem~\ref{t:pythfract} is, indeed, the unique simultaneous discretization of $L$ and $F$ with multiplicity $12$.

\begin{theorem}
\label{t:twelve}
The unique numerical semigroup of multiplicity $12$ that is simultaneously a discretization of the logarithmic monoid $L$ and a discretization of the golden fractal monoid $F$ is the well-tempered numerical semigroup $H$.
\end{theorem}

\begin{proof}
Suppose that $S=\lfloor 12L\rfloor_{\alpha_L}=\lfloor 12F\rfloor_{\alpha_F}$ is a numerical semigroup for some $\alpha_L$ and some $\alpha_F$.
Let $S=\{s_1=0, s_2,s_3,\dots\},$ with $s_i<s_{i+1}$ for all $i \in {\mathbb N}$.
%Observe the table in Appendix~\ref{a:tables}.
From $s_5=\lfloor 12\lambda_5\rfloor_{\alpha_L}=\lfloor
12\varphi_5\rfloor_{\alpha_F}$ we deduce that $\alpha_L\leq 0.8631$
and that $\alpha_F>0.5836$. Consequently, $s_3=19$. Now, $38=19+19$
must be in $S$ and so $\alpha_F>0.8328$. 
%From this bound on $\alpha_F$, the first bound on $\alpha_L$, and the 
%values in Appendix~\ref{a:tables},
From this bound on $\alpha_F$ and the first bound on $\alpha_L$
we deduce that $S=H$.\end{proof}

\section{Conclusion}

Theorem~\ref{t:pythfractef} states that multiplicity $12$ is the largest value for which the discretization of the logarithmic monoid keeps the property that every interval is successively divided using the same ratio (in fact, the golden ratio), and also, it satisfies the property of being half-closed-pipe admissible.
The number $12$ is indeed the number of equal divisions of the octave in classical Western music.

Tuning systems dividing the octave in more than 12 parts, like quarters of
tone or Holdrian commas miss at least one of these properties.

Theorem~\ref{t:twelve} shows that the well-tempered harmonic semigroup $H$ in \eqref{eq:H} is the unique simultaneous discretization of multiplicity $12$ of the unique product-com\-pa\-ti\-ble tempered monoid and the unique nonbisectional fractal monoid of granularity $2$.
In other words, the well-tempered harmonic semigroup $H$ is the unique simultaneous discretization of multiplicity $12$ of the logarithmic monoid and the golden fractal monoid.

For its importance, 
we believe $H$ should be given a name
%$H$ should be given a name 
and we suggest to call it the {\em well-tempered harmonic semigroup}.

\nocite{Sethares}
\addcontentsline{toc}{section}{References}
\bibliographystyle{plain}

\paragraph{Acknowledgment.}
The author would like to thank Julio Fernández, Pilar Bayer, Shalom Eliahou, Alfons Reverté, Pere Casulleras, Marly Cormar, and Felix Gotti for many stimulating and clarifying discussions.
% inspirational 
%motivating, engaging
She would also like to thank the anonymous referees for their encouraging and helpful contributions. Finally she thanks Clifton Callender, Thomas Fiore, and Emmanuel Amiot.
This work was supported by the Catalan Government under grant 2017 SGR 00705 and by the Spanish Ministry of Economy and Competitivity under grant TIN2016-80250-R.

\mut{
\appendix

\section{Algorithms in {\sc Sage}}
\label{a:algorithm}

In this appendix we give an implementation in {\sc Sage} of an algorithm to construct fractal sequences and an algorithm to check whether an ordered sequence is indeed a monoid.

Before constructing a fractal sequence we need to construct its periods. This is what we do with the following algorithm.

\begin{alltt}\small
def fractalperiod(generatorset,periodindex):
  if (periodindex==0):
    return ([0])
  previousperiod=fractalperiod(generatorset,periodindex-1)+[periodindex]
  period=[]
  for i in [0..len(previousperiod)-2]:
    for j in [0..len(generatorset)-1]:
      period+=[1+previousperiod[i]+(generatorset[j]-1)*(previousperiod[i+1]-previousperiod[i])]
  return period
\end{alltt}

Once we have the periods defined, we can construct the fractal sequence up to a certain period. This is what we do with the following algorithm.

\begin{alltt}\small
def fractalsequence(generatorset,maxperiodindex):
  fractalseq=[]
  for i in [0..maxperiodindex]:
    fractalseq+=fractalperiod(generatorset,i)
  return fractalseq
\end{alltt}

Finally, we implemented the following algorithm to check whether an ordered sequence is a monoid up to its maximum given value.

\begin{alltt}\small
def ismonoid(seq):
  if (seq[0] != 0):
    return False
  for i in [1..len(seq)-1]:
    for j in [i..len(seq)-1]:
       if(seq[i]+seq[j] <= seq[len(seq)-1] and seq[i]+seq[j] not in seq):
         print seq[i],"+",seq[j], "does not belong"
         return False
  print "checked up to", seq[len(seq)-1]
  return True
\end{alltt}

As an example, we can construct the fractal monoid in Example~\ref{e:fractmonoidfirst} and check that it is indeed a fractal monoid.
The command
\begin{alltt}\small
fractalsequence([1,1.5,1.75],4)
\end{alltt}
returns

{\footnotesize\tt 
[0, 1, 1.50000000000000, 1.75000000000000, 2.00000000000000, 2.25000000000000, 2.37500000000000, 2.50000000000000, 2.62500000000000, 2.68750000000000, 2.75000000000000, 2.87500000000000, 2.93750000000000, 3.00000000000000, 3.12500000000000, 3.18750000000000, 3.25000000000000, 3.31250000000000, 3.34375000000000, 3.37500000000000, 3.43750000000000, 3.46875000000000, 3.50000000000000, 3.56250000000000, 3.59375000000000, 3.62500000000000, 3.65625000000000, 3.67187500000000, 3.68750000000000, 3.71875000000000, 3.73437500000000, 3.75000000000000, 3.81250000000000, 3.84375000000000, 3.87500000000000, 3.90625000000000, 3.92187500000000, 3.93750000000000, 3.96875000000000, 3.98437500000000, 4.00000000000000, 4.06250000000000, 4.09375000000000, 4.12500000000000, 4.15625000000000, 4.17187500000000, 4.18750000000000, 4.21875000000000, 4.23437500000000, 4.25000000000000, 4.28125000000000, 4.29687500000000, 4.31250000000000, 4.32812500000000, 4.33593750000000, 4.34375000000000, 4.35937500000000, 4.36718750000000, 4.37500000000000, 4.40625000000000, 4.42187500000000, 4.43750000000000, 4.45312500000000, 4.46093750000000, 4.46875000000000, 4.48437500000000, 4.49218750000000, 4.50000000000000, 4.53125000000000, 4.54687500000000, 4.56250000000000, 4.57812500000000, 4.58593750000000, 4.59375000000000, 4.60937500000000, 4.61718750000000, 4.62500000000000, 4.64062500000000, 4.64843750000000, 4.65625000000000, 4.66406250000000, 4.66796875000000, 4.67187500000000, 4.67968750000000, 4.68359375000000, 4.68750000000000, 4.70312500000000, 4.71093750000000, 4.71875000000000, 4.72656250000000, 4.73046875000000, 4.73437500000000, 4.74218750000000, 4.74609375000000, 4.75000000000000, 4.78125000000000, 4.79687500000000, 4.81250000000000, 4.82812500000000, 4.83593750000000, 4.84375000000000, 4.85937500000000, 4.86718750000000, 4.87500000000000, 4.89062500000000, 4.89843750000000, 4.90625000000000, 4.91406250000000, 4.91796875000000, 4.92187500000000, 4.92968750000000, 4.93359375000000, 4.93750000000000, 4.95312500000000, 4.96093750000000, 4.96875000000000, 4.97656250000000, 4.98046875000000, 4.98437500000000, 4.99218750000000, 4.99609375000000]
}
\\
while the command
\\
\begin{alltt}\small
ismonoid(fractalsequence([1,1.5,1.75],4))
\end{alltt}
returns
\begin{alltt}\small
checked up to 4.99609375000000
True
\end{alltt}

Now we can check that the fractal construction in Example~\ref{e:nonfractmonoid} does not give a fractal monoid. Indeed, the commands
\\
\begin{alltt}\small
fractalsequence([1,1+7/12],4)
ismonoid(fractalsequence([1,1+7/12],4))
\end{alltt}
return
\\
{\footnotesize\tt [0, 1, 19/12, 2, 337/144, 31/12, 407/144, 3, 5527/1728, 481/144, 6017/1728, 43/12, 6437/1728, 551/144, 6787/1728, 4, 85345/20736, 7255/1728, 88775/20736, 625/144, 91715/20736, 7745/1728, 94165/20736, 55/12, 96755/20736, 8165/1728, 99205/20736, 695/144, 101305/20736, 8515/1728, 10305520736]}
\begin{alltt}\small
19/12 + 19/12 does not belong
False
\end{alltt}

Finally, we can construct the golden fractal monoid, as defined in Definition~\ref{d:goldenfractal}, and check that is is indeed a fractal monoid, as proved in Theorem~\ref{t:goldenfractal}.
The commands
\begin{alltt}\small
fractalsequence([1,(1+sqrt(5))/2],4)
ismonoid(fractalsequence([1,(1+sqrt(5))/2],4))
\end{alltt}
return
\\
{\footnotesize\tt
 [0, 1, 1/2*sqrt(5) + 1/2, 2, 1/4*(sqrt(5) - 1)\^{ }2 + 2, 1/2*sqrt(5) + 3/2, -1/4*(sqrt(5) - 1)*(sqrt(5) - 3) + 1/2*sqrt(5) + 3/2, 3, 1/8*(sqrt(5) - 1)\^{ }3 + 3, 1/4*(sqrt(5) - 1)\^{ }2 + 3, -1/8*((sqrt(5) - 1)\^{ }2 - 2*sqrt(5) + 2)*(sqrt(5) - 1) + 1/4*(sqrt(5) - 1)\^{ }2 + 3, 1/2*sqrt(5) + 5/2, -1/8*(sqrt(5) - 1)\^{ }2*(sqrt(5) - 3) + 1/2*sqrt(5) + 5/2, -1/4*(sqrt(5) - 1)*(sqrt(5) - 3) + 1/2*sqrt(5) + 5/2, 1/8*((sqrt(5) - 1)*(sqrt(5) - 3) - 2*sqrt(5) + 6)*(sqrt(5) - 1) - 1/4*(sqrt(5) - 1)*(sqrt(5) - 3) + 1/2*sqrt(5) + 5/2, 4, 1/16*(sqrt(5) - 1)\^{ }4 + 4, 1/8*(sqrt(5) - 1)\^{ }3 + 4, 1/8*(sqrt(5) - 1)\^{ }3 - 1/16*((sqrt(5) - 1)\^{ }3 - 2*(sqrt(5) - 1)\^{ }2)*(sqrt(5) - 1) + 4, 1/4*(sqrt(5) - 1)\^{ }2 + 4, -1/16*((sqrt(5) - 1)\^{ }2 - 2*sqrt(5) + 2)*(sqrt(5) - 1)\^{ }2 + 1/4*(sqrt(5) - 1)\^{ }2 + 4, -1/8*((sqrt(5) - 1)\^{ }2 - 2*sqrt(5) + 2)*(sqrt(5) - 1) + 1/4*(sqrt(5) - 1)\^{ }2 + 4, 1/16*(((sqrt(5) - 1)\^{ }2 - 2*sqrt(5) + 2)*(sqrt(5) - 1) - 2*(sqrt(5) - 1)\^{ }2 + 4*sqrt(5) - 4)*(sqrt(5) - 1) - 1/8*((sqrt(5) - 1)\^{ }2 - 2*sqrt(5) + 2)*(sqrt(5) - 1) + 1/4*(sqrt(5) - 1)\^{ }2 + 4, 1/2*sqrt(5) + 7/2, -1/16*(sqrt(5) - 1)\^{ }3*(sqrt(5) - 3) + 1/2*sqrt(5) + 7/2, -1/8*(sqrt(5) - 1)\^{ }2*(sqrt(5) - 3) + 1/2*sqrt(5) + 7/2, -1/8*(sqrt(5) - 1)\^{ }2*(sqrt(5) - 3) + 1/16*((sqrt(5) - 1)\^{ }2*(sqrt(5) - 3) - 2*(sqrt(5) - 1)*(sqrt(5) - 3))*(sqrt(5) - 1) + 1/2*sqrt(5) + 7/2, -1/4*(sqrt(5) - 1)*(sqrt(5) - 3) + 1/2*sqrt(5) + 7/2, 1/16*((sqrt(5) - 1)*(sqrt(5) - 3) - 2*sqrt(5) + 6)*(sqrt(5) - 1)\^{ }2 - 1/4*(sqrt(5) - 1)*(sqrt(5) - 3) + 1/2*sqrt(5) + 7/2, 1/8*((sqrt(5) - 1)*(sqrt(5) - 3) - 2*sqrt(5) + 6)*(sqrt(5) - 1) - 1/4*(sqrt(5) - 1)*(sqrt(5) - 3) + 1/2*sqrt(5) + 7/2, -1/16*(((sqrt(5) - 1)*(sqrt(5) - 3) - 2*sqrt(5) + 6)*(sqrt(5) - 1) - 2*(sqrt(5) - 1)*(sqrt(5) - 3) + 4*sqrt(5) - 12)*(sqrt(5) - 1) + 1/8*((sqrt(5) - 1)*(sqrt(5) - 3) - 2*sqrt(5) + 6)*(sqrt(5) - 1) - 1/4*(sqrt(5) - 1)*(sqrt(5) - 3) + 1/2*sqrt(5) + 7/2]}
\\
{\footnotesize\tt checked up to -1/16*(((sqrt(5) - 1)*(sqrt(5) - 3) - 2*sqrt(5) + 6)*(sqrt(5) - 1) - 2*(sqrt(5) - 1)*(sqrt(5) - 3) + 4*sqrt(5) - 12)*(sqrt(5) - 1) + 1/8*((sqrt(5) - 1)*(sqrt(5) - 3) - 2*sqrt(5) + 6)*(sqrt(5) - 1) - 1/4*(sqrt(5) - 1)*(sqrt(5) - 3) + 1/2*sqrt(5) + 7/2}
\begin{alltt}\small
True
\end{alltt}

\section{Proof of Theorem~\ref{t:goldenfractal}}
\label{a:theorem}

In this appendix we prove Theorem~\ref{t:goldenfractal}, which states that the golden ratio, together with $1$, generates a fractal monoid (the {\em golden fractal monoid}) and that this monoid is, indeed, the unique nonbisectional fractal monoid of granularity $2$. 

\bigskip

\noindent{{\sc Theorem~\ref{t:goldenfractal}}}
  \begin{enumerate}
    \item The period $\{1,\phi\}$ generates a fractal monoid. 
    \item The unique nonbisectional normalized fractal monoid of granularity $2$ is exactly the fractal monoid generated by the period $\{1,\phi\}$. 
  \end{enumerate}

The proof of Theorem~\ref{t:goldenfractal} will be preceded by three lemmas. For the lemmas we need some notation. 
Given $p\in(0,1)$, let $q=1-p$ and define 
$$\begin{array}{rrcl}
 f_0:& \{0\} &\rightarrow& [0,1)\\ 
  & 0 &\mapsto& 0\\\\
  f_\ell:& {\mathbb N}_0\cap [0,2^\ell) &\rightarrow& [0,1)\\ 
&n&\mapsto& \left\{\begin{array}{ll}p f_{\ell -1}(n) & \mbox{ if }n<2^{\ell-1},\\
p +q f_{\ell -1}(n-2^{\ell-1}) & \mbox{ if }n\geq 2^{\ell-1},\end{array}\right. 
\end{array}$$
if $\ell\in{\mathbb N}$. 
Notice that $f_\ell$ depends on the choice of $p$. 
The next lemma is a consequence of the definitions and describes the relationship between the maps $f_\ell$ and fractal monoids. 

\begin{lemma}\label{l:struct}
If a tempered monoid $M=\{\mu_1,\mu_2,\mu_3,\dots\}$ with $\mu_i<\mu_{i+1}$ is fractal and has first period equal to 
$\{1,1+p\}$, then 
\begin{itemize}
  \item the elements of $M$ are exactly $\mu_i=\lfloor\log_2(i)\rfloor+f_{\lfloor\log_2(i)\rfloor}(i-2^{\lfloor\log_2(i)\rfloor})$,
\item 
  the $\ell$th period of $F$ is $\{\ell+f_\ell(0),\ell+f_\ell(1),\ell+f_\ell(2),\ell+f_\ell(3),\dots,\ell+f_\ell(2^\ell-1)\}$, with $\ell+f_\ell(0)<\ell+f_\ell(1)<\ell+f_\ell(2)<\ell+f_\ell(3)<\dots<\ell+f_\ell(2^\ell-1)$. 
  \end{itemize}
\end{lemma}

The next lemma states a simple result that will be used several times in 
the proofs of the following lemmas. It can be proved by induction. 

\begin{lemma}If $\ell,n\in{\mathbb N}_0$ with $\ell\geq 1$ and $n<2^{\ell-1}$, then $f_{\ell-1}(n)=f_{\ell}(2n)$. 
\end{lemma}

\mut{
\begin{proof}
  If $\ell=1$ the result is obvious. For $\ell>1$ we will proceed by induction. 
  \begin{eqnarray*}
    f_\ell(2n)&=&\left\{\begin{array}{ll}p f_{\ell -1}(2n) & \mbox{ if }2n<2^{\ell-1},\\p +q f_{\ell -1}(2n-2^{\ell-1}) & \mbox{ if }2n\geq 2^{\ell-1},\end{array}\right. 
    \\&=&\left\{\begin{array}{ll}p f_{\ell -1}(2n) & \mbox{ if }n<2^{\ell-2},\\p +q f_{\ell -1}(2n-2^{\ell-1}) & \mbox{ if }n\geq 2^{\ell-2},\end{array}\right. 
  \end{eqnarray*}
  Now, by the induction hypothesis, this equals 
  \begin{eqnarray*}
    \phantom{f_\ell(2n)}&\phantom{=}&\left\{\begin{array}{ll}p f_{\ell -2}(n) & \mbox{ if }n<2^{\ell-2},\\p +q f_{\ell -2}(n-2^{\ell-2}) & \mbox{ if }n\geq 2^{\ell-2},\end{array}\right.\\
&=&f_{\ell-1}(n). 
  \end{eqnarray*}
\end{proof}
}

Recall that the golden ratio is $\phi=\frac{1+\sqrt{5}}{2}$. 
In the next lemma we take $p=\phi-1=\frac{-1+\sqrt{5}}{2}=0.618033\dots$. 
 We use that, in this case, $p^2=1-p$, $1=p+p^2$ and $2p=1+p^3$. 

\begin{lemma}\label{l:p}
  Let $p=\phi-1$, $q=1-p$, and let $\ell,n\in{\mathbb N}_0$ with 
  $\ell>0$ and $n<2^\ell$. The sum 
  $p+f_\ell(n)$ is one of 
\begin{itemize}
  \item $f_{\ell+1}(c)$ for some $0\leq c<2^{\ell+1}$,
  \item $1+pf_{\ell+1}(d)$ for some $0\leq d<2^{\ell+1}$. 
\end{itemize}
\end{lemma}

\begin{proof}

    If $\ell=1$ then $n$ is either $0$ or $1$ and $p+f_1(0)=p+0=p=f_1(1)$, while $p+f_1(1)=p+p=1+p^3=1+pf_2(1)$. 
Suppose that $\ell>1$ and assume that the lemma is true for $\ell-1$.  

If $n<2^{\ell-2}$, then $p+f_\ell(n)=p+p f_{\ell -1}(n)=p+p^2 f_{\ell-2}(n)=f_{\ell-1}(n+2^{\ell-2})=f_{\ell+1}(2^2n+2^{\ell})$, with $0\leq 2^2n+2^{\ell}<2^{\ell+1}$. 

If $2^{\ell-2}\leq n<2^{\ell-1}$, then $p+f_\ell(n)=p+pf_{\ell-1}(n)=p+p^2+p^3 f_{\ell-2}(n-2^{\ell-2})=1+p^3f_{\ell-2}(n-2^{\ell-2})=1+pf_{\ell}(n-2^{\ell-2})=1+pf_{\ell+1}(2n-2^{\ell-1})  $, with $0\leq 2n-2^{\ell-1}<2^{\ell+1}$. 

If $n\geq 2^{\ell-1}$, then $p+f_\ell(n)=2p +p^2 f_{\ell -1}(n-2^{\ell-1}) = 1+p^3 +p^2 f_{\ell -1}(n-2^{\ell-1}) = 1+p^2(p+f_{\ell-1}(n-2^{\ell-1}))$ with $n-2^{\ell-1}<2^{\ell-1}$. 
By the induction hypothesis, this either equals $1+p^2f_{\ell}(c')=1+pf_{\ell+1}(c')$ with $c'<2^\ell$, and so, with $c'<2^{\ell+1}$, or $1+p^2(1+pf_{\ell}(d'))=1+p(p+p^2f_{\ell}(d'))=1+pf_{\ell+1}(d'+2^\ell)$ with $d'<2^\ell$, and so, with $d'+2^\ell<2^{\ell+1}$. 

\end{proof}

\begin{lemma}\label{l:key}
Let $p=\phi-1$ and $q=1-p$. 
If $i,j,a,b\in{\mathbb N}_0$, with $0\leq a<2^i$ and $0\leq b<2^j$, then 
$f_i(a)+f_j(b)$
is one of 
\begin{itemize}
  \item $f_{i+j}(c)$ for some $0\leq c<2^{i+j}$,
  \item $1+f_{i+j+1}(d)$ for some $0\leq d<2^{i+j+1}$. 
    \end{itemize}
\end{lemma}

\begin{proof}
We proceed by induction on $i+j$. 
If $i+j$ equals $0$ then the result is obvious. 
Suppose that $i+j>0$. 
If one of $i$ and $j$ is $0$ then the result is also obvious,
so, we can assume that both $i$ and $j$ are non-zero. 
On one hand,
$f_i(a)=pf_{i-1}(a')$ or 
$f_i(a)=p+qf_{i-1}(a')$ for some $a'<2^{i-1}$. 
On the other hand,
$f_j(b)=pf_{j-1}(b')$ or 
$f_j(b)=p+qf_{j-1}(b')$ for some $b'<2^{j-1}$. 
So, one of the next cases holds. 
\begin{enumerate}
\item $f_i(a)+f_j(b)=pf_{i-1}(a')+pf_{j-1}(b')=p(f_{i-1}(a')+f_{j-1}(b'))$
  for some $a'<2^{i-1}$ and some $b'<2^{j-1}$. 
  By the induction hypothesis, this equals one of 
  \begin{itemize}
  \item $pf_{i+j-2}(c)$ for some $0\leq c<2^{i+j-2}$,
  \item $p(1+f_{i+j-1}(d))$ for some $0\leq d<2^{i+j-1}$. 
  \end{itemize}
  On one hand, $pf_{i+j-2}(c)=f_{i+j-1}(c)=f_{i+j}(2c)$, with $2c<2^{i+j}$. 
  On the other hand, $p(1+f_{i+j-1}(d))$ can be either 
  $p(1+pf_{i+j-2}(d'))$ or $p(1+p+qf_{i+j-2}(d'))$ for some $d'<2^{i+j-2}$. 
  But $p(1+pf_{i+j-2}(d'))=p+p^2f_{i+j-2}(d')=f_{i+j-1}(d'+2^{i+j-2})=f_{i+j}(2d'+2^{i+j-1})$ with $2d'+2^{i+j-1}<2^{i+j}$,
  while $p(1+p+qf_{i+j-2}(d'))=p+p^2+p^3f_{i+j-2}(d')=1+f_{i+j+1}(d')$, with $d'<2^{i+j+1}$. 

\item $f_i(a)+f_j(b)=pf_{i-1}(a')+p+qf_{j-1}(b')$
  for some $a'<2^{i-1}$ and some $b'<2^{j-1}$. 
  If $i=1$, then $a'=0$ and $f_i(a)+f_j(b)=p+qf_{j-1}(b')=f_j(b'+2^{j-1})=f_{i+j}(2^ib'+2^{i+j-1})$, with $2^ib'+2^{i+j-1}<2^{i+j}$. So, we can assume $i>1$. 
  Now, by the definition of $f$, the sum $f_i(a)+f_j(b)$ equals one of 
  \begin{itemize}
  \item $p^2f_{i-2}(a')+p+qf_{j-1}(b')=p+p^2(f_{i-2}(a')+f_{j-1}(b'))$ if $a'<2^{i-2}$,
  \item 
 $p(p+qf_{i-2}(a'-2^{i-2}))+p+qf_{j-1}(b')=1+p^2(pf_{i-2}(a'-2^{i-2})+f_{j-1}(b'))$
    if $a'\geq 2^{i-2}$. 
  \end{itemize}
  The first sum, by the induction hypothesis, is either $p+q(f_{i+j-3}(c'))=f_{i+j-2}(c'+2^{i+j-3})=f_{i+j}(2^2c'+2^{i+j-1})$ with $c'<2^{i+j-3}$ and so, with $2^2c'+2^{i+j-1}<2^{i+j}$, or $p+q(1+f_{i+j-2}(d'))=1+p^2f_{i+j-2}(d')=1+f_{i+j}(d')=1+f_{i+1+1}(2d')$, with $d'<2^{i+j-2}$ (and so with $2d'<2^{i+j+1}$). 

    To analyze the second sum, notice that $pf_{i-2}(a'-2^{i-2})<p$ and, so, $pf_{i-2}(a'-2^{i-2})+f_{j-1}(b')<1+p$. Now, taking this into account, and applying the induction hypothesis, we have 
that 
$1+p^2(pf_{i-2}(a'-2^{i-2})+f_{j-1}(b'))$ is either 
$1+p^2(f_{i+j-2}(c'))=1+f_{i+j}(c')=1+f_{i+j+1}(2c')$ for some $c'<2^{i+j-2}$ (and so $2c'<2^{i+j+1}$) 
or 
$1+p^2(1+pf_{i+j-2}(c''))=1+p(p+qf_{i+j-2}(c''))=1+pf_{i+j-1}(c''+2^{i+j-2})=1+f_{i+j}(c''+2^{i+j-2})=1+f_{i+j+1}(2c''+2^{i+j-1})$ for some $c''<2^{i+j-2}$ (and so $2c''+2^{i+j-1}<2^{i+j+1}$). 
  
\item $f_i(a)+f_j(b)=p+qf_{i-1}(a')+pf_{j-1}(b')$ for some $a'<2^{i-1}$ and some $b'<2^{j-1}$. This case can be proved as the previous point. 
\item $f_i(a)+f_j(b)=p+qf_{i-1}(a')+p+qf_{j-1}(b')=2p+q(f_{i-1}(a')+f_{j-1}(b'))$ for some $a'<2^{i-1}$ and some $b'<2^{j-1}$. 
By the induction hypothesis, this equals one of 
\begin{itemize}
\item $2p+qf_{i+j-2}(c')$ for some $0\leq c'<2^{i+j-2}$,
\item $2p+q(1+f_{i+j-1}(d'))$ for some $0\leq d'<2^{i+j-1}$. 
\end{itemize}
On one hand, $2p+qf_{i+j-2}(c')=p+f_{i+j-1}(c'+2^{i+j-2})$ and the result follows by Lemma~\ref{l:p}. 

On the other hand,
$2p+q(1+f_{i+j-1}(d'))=1+p+qf_{i+j-1}(d')=1+f_{i+j}(d'+2^{i+j-1})=1+f_{i+j+1}(2d'+2^{i+j})$ with $2d'+2^{i+j}<2^{i+j+1}$. 
\end{enumerate}
\end{proof}

Now we are ready to prove Theorem~\ref{t:goldenfractal}. 

\begin{proof}
  \begin{enumerate}
  \item One needs to see that, for the case when $p=\phi-1$, the set described in Lemma~\ref{l:struct} is closed under addition. This is a direct consequence of Lemma~\ref{l:key}. 
    \item 
      Suppose that $M$ is a normalized nonbisectional fractal monoid of granularity $2$. Then its first period is $\{1,1+p\}$ for some $p$ with $0<p<1$. 
      By Lemma~\ref{l:struct}, the second and third periods of the tempered monoid must be 
  $$\begin{array}{l}
    \{2,2+p^2,2+p,2+2p-p^2\},\\
      \{3,3+p^3,3+p^2,3+2p^2-p^3,3+p,3+p+p^2-p^3,3+2p-p^2,3+3p-3p^2+p^3\},\\
  \end{array}$$
  where the elements in each period are presented in increasing order. 
  
  If $M$ is nonbisectional then $p\neq 0.5$. 
  For $M$ to be closed under addition, it must hold that $(1+p)+(1+p)=2+2p\in M$. 
  
  If $p<0.5$, then $2+p<2+2p<3$. Looking at the elements of the second period, one can deduce that $2+2p=2+2p-p^2$, leading to $p=0$, which is out of the range. 

  Then, $p$ must be larger than $0.5$. In this case, $3<2+2p<3+p$. Looking at the elements of the third period, one can deduce that $2+2p$ is either 
$3+p^3$, $3+p^2$, or $3+2p^2-p^3$. This leads to the equations $p^3-2p+1=(p^2+p-1)(p-1)=0$, $p^2-2p+1=(p-1)^2=0$ or $p^3-2p^2+2p-1=(p^2-p+1)(p-1)=0$. Among these equations, the unique having a real solution in the range $0.5 < p < 1$ is the first one, being the solution $p=\frac{-1+\sqrt{5}}{2}=\phi-1$. Hence, $1+p=\phi$. 
  \end{enumerate}
\end{proof}

\newpage
\section{Tables}
\label{a:tables}

\rotatebox{90}{\resizebox{.8\textheight}{!}{\begin{tabular}{ccccc}
$\begin{array}{ll}
9\lambda_{1}=0.0000	& 9\varphi_{1}=0.0000\\
9\lambda_{2}=9.0000	& 9\varphi_{2}=9.0000\\
9\lambda_{3}=14.2647	& 9\varphi_{3}=14.5623\\
9\lambda_{4}=18.0000	& 9\varphi_{4}=18.0000\\
9\lambda_{5}=20.8974	& 9\varphi_{5}=21.4377\\
9\lambda_{6}=23.2647	& 9\varphi_{6}=23.5623\\
9\lambda_{7}=25.2662	& 9\varphi_{7}=25.6869\\
9\lambda_{8}=27.0000	& 9\varphi_{8}=27.0000\\
9\lambda_{9}=28.5293	& 9\varphi_{9}=29.1246\\
9\lambda_{10}=29.8974	& 9\varphi_{10}=30.4377\\
9\lambda_{11}=31.1349	& 9\varphi_{11}=31.7508\\
9\lambda_{12}=32.2647	& 9\varphi_{12}=32.5623\\
9\lambda_{13}=33.3040	& 9\varphi_{13}=33.8754\\
9\lambda_{14}=34.2662	& 9\varphi_{14}=34.6869\\
9\lambda_{15}=35.1620	& 9\varphi_{15}=35.4984\\
9\lambda_{16}=36.0000	& 9\varphi_{16}=36.0000\\
9\lambda_{17}=36.7872	& 9\varphi_{17}=37.3131\\
9\lambda_{18}=37.5293	& 9\varphi_{18}=38.1246\\
9\lambda_{19}=38.2313	& 9\varphi_{19}=38.9361\\
9\lambda_{20}=38.8974	& 9\varphi_{20}=39.4377\\
9\lambda_{21}=39.5309	& 9\varphi_{21}=40.2492\\
9\lambda_{22}=40.1349	& 9\varphi_{22}=40.7508\\
9\lambda_{23}=40.7121	& 9\varphi_{23}=41.2523\\
9\lambda_{24}=41.2647	& 9\varphi_{24}=41.5623\\
9\lambda_{25}=41.7947	& 9\varphi_{25}=42.3738\\
9\lambda_{26}=42.3040	& 9\varphi_{26}=42.8754\\
9\lambda_{27}=42.7940	& 9\varphi_{27}=43.3769\\
9\lambda_{28}=43.2662	& 9\varphi_{28}=43.6869\\
9\lambda_{29}=43.7218	& 9\varphi_{29}=44.1885\\
9\lambda_{30}=44.1620	& 9\varphi_{30}=44.4984\\
9\lambda_{31}=44.5878	& 9\varphi_{31}=44.8084\\
9\lambda_{32}=45.0000	& 9\varphi_{32}=45.0000\\
9\lambda_{33}=45.3995	& 9\varphi_{33}=45.8115\\
9\lambda_{34}=45.7872	& 9\varphi_{34}=46.3131\\
9\lambda_{35}=46.1635	& 9\varphi_{35}=46.8146\\
9\lambda_{36}=46.5293	& 9\varphi_{36}=47.1246\\
9\lambda_{37}=46.8851	& 9\varphi_{37}=47.6262\\
9\lambda_{38}=47.2313	& 9\varphi_{38}=47.9361\\
9\lambda_{39}=47.5686	& 9\varphi_{39}=48.2461\\
9\lambda_{40}=47.8974	& 9\varphi_{40}=48.4377\\
9\lambda_{41}=48.2180	& 9\varphi_{41}=48.9392\\
9\lambda_{42}=48.5309	& 9\varphi_{42}=49.2492\\
9\lambda_{43}=48.8364	& 9\varphi_{43}=49.5592\\
9\lambda_{44}=49.1349	& 9\varphi_{44}=49.7508\\
9\lambda_{45}=49.4267	& 9\varphi_{45}=50.0608\\
9\lambda_{46}=49.7121	& 9\varphi_{46}=50.2523\\
9\lambda_{47}=49.9913	& 9\varphi_{47}=50.4439\\
9\lambda_{48}=50.2647	& 9\varphi_{48}=50.5623\\
9\lambda_{49}=50.5324	& 9\varphi_{49}=51.0639\\
9\lambda_{50}=50.7947	& 9\varphi_{50}=51.3738\\
\end{array}$
&
$\begin{array}{ll}
11\lambda_{1}=0.0000	& 11\varphi_{1}=0.0000\\
11\lambda_{2}=11.0000	& 11\varphi_{2}=11.0000\\
11\lambda_{3}=17.4346	& 11\varphi_{3}=17.7984\\
11\lambda_{4}=22.0000	& 11\varphi_{4}=22.0000\\
11\lambda_{5}=25.5412	& 11\varphi_{5}=26.2016\\
11\lambda_{6}=28.4346	& 11\varphi_{6}=28.7984\\
11\lambda_{7}=30.8809	& 11\varphi_{7}=31.3951\\
11\lambda_{8}=33.0000	& 11\varphi_{8}=33.0000\\
11\lambda_{9}=34.8692	& 11\varphi_{9}=35.5967\\
11\lambda_{10}=36.5412	& 11\varphi_{10}=37.2016\\
11\lambda_{11}=38.0537	& 11\varphi_{11}=38.8065\\
11\lambda_{12}=39.4346	& 11\varphi_{12}=39.7984\\
11\lambda_{13}=40.7048	& 11\varphi_{13}=41.4033\\
11\lambda_{14}=41.8809	& 11\varphi_{14}=42.3951\\
11\lambda_{15}=42.9758	& 11\varphi_{15}=43.3870\\
11\lambda_{16}=44.0000	& 11\varphi_{16}=44.0000\\
11\lambda_{17}=44.9621	& 11\varphi_{17}=45.6049\\
11\lambda_{18}=45.8692	& 11\varphi_{18}=46.5967\\
11\lambda_{19}=46.7272	& 11\varphi_{19}=47.5886\\
11\lambda_{20}=47.5412	& 11\varphi_{20}=48.2016\\
11\lambda_{21}=48.3155	& 11\varphi_{21}=49.1935\\
11\lambda_{22}=49.0537	& 11\varphi_{22}=49.8065\\
11\lambda_{23}=49.7592	& 11\varphi_{23}=50.4195\\
11\lambda_{24}=50.4346	& 11\varphi_{24}=50.7984\\
11\lambda_{25}=51.0824	& 11\varphi_{25}=51.7902\\
11\lambda_{26}=51.7048	& 11\varphi_{26}=52.4033\\
11\lambda_{27}=52.3038	& 11\varphi_{27}=53.0163\\
11\lambda_{28}=52.8809	& 11\varphi_{28}=53.3951\\
11\lambda_{29}=53.4378	& 11\varphi_{29}=54.0081\\
11\lambda_{30}=53.9758	& 11\varphi_{30}=54.3870\\
11\lambda_{31}=54.4962	& 11\varphi_{31}=54.7659\\
11\lambda_{32}=55.0000	& 11\varphi_{32}=55.0000\\
11\lambda_{33}=55.4883	& 11\varphi_{33}=55.9919\\
11\lambda_{34}=55.9621	& 11\varphi_{34}=56.6049\\
11\lambda_{35}=56.4221	& 11\varphi_{35}=57.2179\\
11\lambda_{36}=56.8692	& 11\varphi_{36}=57.5967\\
11\lambda_{37}=57.3040	& 11\varphi_{37}=58.2098\\
11\lambda_{38}=57.7272	& 11\varphi_{38}=58.5886\\
11\lambda_{39}=58.1394	& 11\varphi_{39}=58.9675\\
11\lambda_{40}=58.5412	& 11\varphi_{40}=59.2016\\
11\lambda_{41}=58.9331	& 11\varphi_{41}=59.8146\\
11\lambda_{42}=59.3155	& 11\varphi_{42}=60.1935\\
11\lambda_{43}=59.6889	& 11\varphi_{43}=60.5724\\
11\lambda_{44}=60.0537	& 11\varphi_{44}=60.8065\\
11\lambda_{45}=60.4104	& 11\varphi_{45}=61.1854\\
11\lambda_{46}=60.7592	& 11\varphi_{46}=61.4195\\
11\lambda_{47}=61.1005	& 11\varphi_{47}=61.6537\\
11\lambda_{48}=61.4346	& 11\varphi_{48}=61.7984\\
11\lambda_{49}=61.7618	& 11\varphi_{49}=62.4114\\
11\lambda_{50}=62.0824	& 11\varphi_{50}=62.7902\\
\end{array}$
&
$\begin{array}{ll}
12\lambda_{1}=0.0000	& 12\varphi_{1}=0.0000\\
12\lambda_{2}=12.0000	& 12\varphi_{2}=12.0000\\
12\lambda_{3}=19.0196	& 12\varphi_{3}=19.4164\\
12\lambda_{4}=24.0000	& 12\varphi_{4}=24.0000\\
12\lambda_{5}=27.8631	& 12\varphi_{5}=28.5836\\
12\lambda_{6}=31.0196	& 12\varphi_{6}=31.4164\\
12\lambda_{7}=33.6883	& 12\varphi_{7}=34.2492\\
12\lambda_{8}=36.0000	& 12\varphi_{8}=36.0000\\
12\lambda_{9}=38.0391	& 12\varphi_{9}=38.8328\\
12\lambda_{10}=39.8631	& 12\varphi_{10}=40.5836\\
12\lambda_{11}=41.5132	& 12\varphi_{11}=42.3344\\
12\lambda_{12}=43.0196	& 12\varphi_{12}=43.4164\\
12\lambda_{13}=44.4053	& 12\varphi_{13}=45.1672\\
12\lambda_{14}=45.6883	& 12\varphi_{14}=46.2492\\
12\lambda_{15}=46.8827	& 12\varphi_{15}=47.3313\\
12\lambda_{16}=48.0000	& 12\varphi_{16}=48.0000\\
12\lambda_{17}=49.0496	& 12\varphi_{17}=49.7508\\
12\lambda_{18}=50.0391	& 12\varphi_{18}=50.8328\\
12\lambda_{19}=50.9751	& 12\varphi_{19}=51.9149\\
12\lambda_{20}=51.8631	& 12\varphi_{20}=52.5836\\
12\lambda_{21}=52.7078	& 12\varphi_{21}=53.6656\\
12\lambda_{22}=53.5132	& 12\varphi_{22}=54.3344\\
12\lambda_{23}=54.2827	& 12\varphi_{23}=55.0031\\
12\lambda_{24}=55.0196	& 12\varphi_{24}=55.4164\\
12\lambda_{25}=55.7263	& 12\varphi_{25}=56.4984\\
12\lambda_{26}=56.4053	& 12\varphi_{26}=57.1672\\
12\lambda_{27}=57.0587	& 12\varphi_{27}=57.8359\\
12\lambda_{28}=57.6883	& 12\varphi_{28}=58.2492\\
12\lambda_{29}=58.2958	& 12\varphi_{29}=58.9180\\
12\lambda_{30}=58.8827	& 12\varphi_{30}=59.3313\\
12\lambda_{31}=59.4504	& 12\varphi_{31}=59.7446\\
12\lambda_{32}=60.0000	& 12\varphi_{32}=60.0000\\
12\lambda_{33}=60.5327	& 12\varphi_{33}=61.0820\\
12\lambda_{34}=61.0496	& 12\varphi_{34}=61.7508\\
12\lambda_{35}=61.5514	& 12\varphi_{35}=62.4195\\
12\lambda_{36}=62.0391	& 12\varphi_{36}=62.8328\\
12\lambda_{37}=62.5134	& 12\varphi_{37}=63.5016\\
12\lambda_{38}=62.9751	& 12\varphi_{38}=63.9149\\
12\lambda_{39}=63.4248	& 12\varphi_{39}=64.3282\\
12\lambda_{40}=63.8631	& 12\varphi_{40}=64.5836\\
12\lambda_{41}=64.2906	& 12\varphi_{41}=65.2523\\
12\lambda_{42}=64.7078	& 12\varphi_{42}=65.6656\\
12\lambda_{43}=65.1152	& 12\varphi_{43}=66.0789\\
12\lambda_{44}=65.5132	& 12\varphi_{44}=66.3344\\
12\lambda_{45}=65.9022	& 12\varphi_{45}=66.7477\\
12\lambda_{46}=66.2827	& 12\varphi_{46}=67.0031\\
12\lambda_{47}=66.6551	& 12\varphi_{47}=67.2585\\
12\lambda_{48}=67.0196	& 12\varphi_{48}=67.4164\\
12\lambda_{49}=67.3765	& 12\varphi_{49}=68.0851\\
12\lambda_{50}=67.7263	& 12\varphi_{50}=68.4984\\
\end{array}$
&
$\begin{array}{ll}
13\lambda_{1}=0.0000	& 13\varphi_{1}=0.0000\\
13\lambda_{2}=13.0000	& 13\varphi_{2}=13.0000\\
13\lambda_{3}=20.6045	& 13\varphi_{3}=21.0344\\
13\lambda_{4}=26.0000	& 13\varphi_{4}=26.0000\\
13\lambda_{5}=30.1851	& 13\varphi_{5}=30.9656\\
13\lambda_{6}=33.6045	& 13\varphi_{6}=34.0344\\
13\lambda_{7}=36.4956	& 13\varphi_{7}=37.1033\\
13\lambda_{8}=39.0000	& 13\varphi_{8}=39.0000\\
13\lambda_{9}=41.2090	& 13\varphi_{9}=42.0689\\
13\lambda_{10}=43.1851	& 13\varphi_{10}=43.9656\\
13\lambda_{11}=44.9726	& 13\varphi_{11}=45.8622\\
13\lambda_{12}=46.6045	& 13\varphi_{12}=47.0344\\
13\lambda_{13}=48.1057	& 13\varphi_{13}=48.9311\\
13\lambda_{14}=49.4956	& 13\varphi_{14}=50.1033\\
13\lambda_{15}=50.7896	& 13\varphi_{15}=51.2755\\
13\lambda_{16}=52.0000	& 13\varphi_{16}=52.0000\\
13\lambda_{17}=53.1370	& 13\varphi_{17}=53.8967\\
13\lambda_{18}=54.2090	& 13\varphi_{18}=55.0689\\
13\lambda_{19}=55.2231	& 13\varphi_{19}=56.2411\\
13\lambda_{20}=56.1851	& 13\varphi_{20}=56.9656\\
13\lambda_{21}=57.1001	& 13\varphi_{21}=58.1378\\
13\lambda_{22}=57.9726	& 13\varphi_{22}=58.8622\\
13\lambda_{23}=58.8063	& 13\varphi_{23}=59.5867\\
13\lambda_{24}=59.6045	& 13\varphi_{24}=60.0344\\
13\lambda_{25}=60.3701	& 13\varphi_{25}=61.2067\\
13\lambda_{26}=61.1057	& 13\varphi_{26}=61.9311\\
13\lambda_{27}=61.8135	& 13\varphi_{27}=62.6556\\
13\lambda_{28}=62.4956	& 13\varphi_{28}=63.1033\\
13\lambda_{29}=63.1538	& 13\varphi_{29}=63.8278\\
13\lambda_{30}=63.7896	& 13\varphi_{30}=64.2755\\
13\lambda_{31}=64.4046	& 13\varphi_{31}=64.7233\\
13\lambda_{32}=65.0000	& 13\varphi_{32}=65.0000\\
13\lambda_{33}=65.5771	& 13\varphi_{33}=66.1722\\
13\lambda_{34}=66.1370	& 13\varphi_{34}=66.8967\\
13\lambda_{35}=66.6807	& 13\varphi_{35}=67.6211\\
13\lambda_{36}=67.2090	& 13\varphi_{36}=68.0689\\
13\lambda_{37}=67.7229	& 13\varphi_{37}=68.7933\\
13\lambda_{38}=68.2231	& 13\varphi_{38}=69.2411\\
13\lambda_{39}=68.7102	& 13\varphi_{39}=69.6888\\
13\lambda_{40}=69.1851	& 13\varphi_{40}=69.9656\\
13\lambda_{41}=69.6482	& 13\varphi_{41}=70.6900\\
13\lambda_{42}=70.1001	& 13\varphi_{42}=71.1378\\
13\lambda_{43}=70.5414	& 13\varphi_{43}=71.5855\\
13\lambda_{44}=70.9726	& 13\varphi_{44}=71.8622\\
13\lambda_{45}=71.3941	& 13\varphi_{45}=72.3100\\
13\lambda_{46}=71.8063	& 13\varphi_{46}=72.5867\\
13\lambda_{47}=72.2097	& 13\varphi_{47}=72.8634\\
13\lambda_{48}=72.6045	& 13\varphi_{48}=73.0344\\
13\lambda_{49}=72.9912	& 13\varphi_{49}=73.7589\\
13\lambda_{50}=73.3701	& 13\varphi_{50}=74.2067\\
\end{array}$
&
$\begin{array}{ll}
18\lambda_{1}=0.0000	& 18\varphi_{1}=0.0000\\
18\lambda_{2}=18.0000	& 18\varphi_{2}=18.0000\\
18\lambda_{3}=28.5293	& 18\varphi_{3}=29.1246\\
18\lambda_{4}=36.0000	& 18\varphi_{4}=36.0000\\
18\lambda_{5}=41.7947	& 18\varphi_{5}=42.8754\\
18\lambda_{6}=46.5293	& 18\varphi_{6}=47.1246\\
18\lambda_{7}=50.5324	& 18\varphi_{7}=51.3738\\
18\lambda_{8}=54.0000	& 18\varphi_{8}=54.0000\\
18\lambda_{9}=57.0587	& 18\varphi_{9}=58.2492\\
18\lambda_{10}=59.7947	& 18\varphi_{10}=60.8754\\
18\lambda_{11}=62.2698	& 18\varphi_{11}=63.5016\\
18\lambda_{12}=64.5293	& 18\varphi_{12}=65.1246\\
18\lambda_{13}=66.6079	& 18\varphi_{13}=67.7508\\
18\lambda_{14}=68.5324	& 18\varphi_{14}=69.3738\\
18\lambda_{15}=70.3240	& 18\varphi_{15}=70.9969\\
18\lambda_{16}=72.0000	& 18\varphi_{16}=72.0000\\
18\lambda_{17}=73.5743	& 18\varphi_{17}=74.6262\\
18\lambda_{18}=75.0587	& 18\varphi_{18}=76.2492\\
18\lambda_{19}=76.4627	& 18\varphi_{19}=77.8723\\
18\lambda_{20}=77.7947	& 18\varphi_{20}=78.8754\\
18\lambda_{21}=79.0617	& 18\varphi_{21}=80.4984\\
18\lambda_{22}=80.2698	& 18\varphi_{22}=81.5016\\
18\lambda_{23}=81.4241	& 18\varphi_{23}=82.5047\\
18\lambda_{24}=82.5293	& 18\varphi_{24}=83.1246\\
18\lambda_{25}=83.5894	& 18\varphi_{25}=84.7477\\
18\lambda_{26}=84.6079	& 18\varphi_{26}=85.7508\\
18\lambda_{27}=85.5880	& 18\varphi_{27}=86.7539\\
18\lambda_{28}=86.5324	& 18\varphi_{28}=87.3738\\
18\lambda_{29}=87.4437	& 18\varphi_{29}=88.3769\\
18\lambda_{30}=88.3240	& 18\varphi_{30}=88.9969\\
18\lambda_{31}=89.1755	& 18\varphi_{31}=89.6168\\
18\lambda_{32}=90.0000	& 18\varphi_{32}=90.0000\\
18\lambda_{33}=90.7991	& 18\varphi_{33}=91.6231\\
18\lambda_{34}=91.5743	& 18\varphi_{34}=92.6262\\
18\lambda_{35}=92.3271	& 18\varphi_{35}=93.6293\\
18\lambda_{36}=93.0587	& 18\varphi_{36}=94.2492\\
18\lambda_{37}=93.7702	& 18\varphi_{37}=95.2523\\
18\lambda_{38}=94.4627	& 18\varphi_{38}=95.8723\\
18\lambda_{39}=95.1372	& 18\varphi_{39}=96.4922\\
18\lambda_{40}=95.7947	& 18\varphi_{40}=96.8754\\
18\lambda_{41}=96.4359	& 18\varphi_{41}=97.8785\\
18\lambda_{42}=97.0617	& 18\varphi_{42}=98.4984\\
18\lambda_{43}=97.6728	& 18\varphi_{43}=99.1184\\
18\lambda_{44}=98.2698	& 18\varphi_{44}=99.5016\\
18\lambda_{45}=98.8534	& 18\varphi_{45}=100.1215\\
18\lambda_{46}=99.4241	& 18\varphi_{46}=100.5047\\
18\lambda_{47}=99.9826	& 18\varphi_{47}=100.8878\\
18\lambda_{48}=100.5293	& 18\varphi_{48}=101.1246\\
18\lambda_{49}=101.0648	& 18\varphi_{49}=102.1277\\
18\lambda_{50}=101.5894	& 18\varphi_{50}=102.7477\\
\end{array}$
\end{tabular}}}
}

\end{document}